\documentclass[reqno,11pt]{amsart}

\usepackage{graphicx}

\usepackage{latexsym}
\usepackage{amssymb}
\usepackage{amsfonts}
\usepackage{amscd}
\usepackage{amsmath,amsthm}

\theoremstyle{plain} 

\newtheorem{lemma}{Lemma}[section]
\newtheorem{proposition}[lemma]{Proposition}
\newtheorem{theorem}[lemma]{Theorem}
\newtheorem{corollary}[lemma]{Corollary}

\theoremstyle{definition}

\newtheorem{example}[lemma]{Example}
\newtheorem{definition}[lemma]{Definition}

\theoremstyle{remark}

\newtheorem{remark}[lemma]{Remark}

\numberwithin{equation}{section}

\def\NN{{\mathbb N}}

\def\PP{{\mathbb P}}

\def\ZZ{{\mathbb Z}}

\def\0ol{{\bar 0}}
\def\1ol{{\bar 1}}
\def\2ol{{\bar 2}}
\def\ol2{{\bar 2}}
\def\3ol{{\bar 3}}
\def\4ol{{\bar 4}}
\def\5ol{{\bar 5}}
\def\6ol{{\bar 6}}
\def\7ol{{\bar 7}}
\def\8ol{{\bar 8}}
\def\9ol{{\bar 9}}

\def\bold0{{\bf 0}}
\def\bold1{{\bf 1}}
\def\bold2{{\bf 2}} 
\def\bold3{{\bf  3}}
\def\bold4{{\bf 4}}
\def\bold5{{\bf 5}}
\def\bold6{{\bf 6}}
\def\bold7{{\bf 7}}
\def\bold8{{\bf 8}}
\def\bold9{{\bf 9}}

\def\P2Skly{\PP^2_{Skly}}

\def\Alt{\operatorname {Alt}}

\def\Ext{\operatorname {Ext}}

\def\GL{\operatorname {GL}}

\def\Hom{\operatorname {Hom}}

\def\pd{{\operatorname {\partial}}}

\def\PGL{\operatorname {PGL}}

\def\sgn{\operatorname {sgn}}

\def\Aut{\operatorname{Aut}}

\def\det{\operatorname{det}}

\def\dim{\operatorname{dim}}

\def\Ext{\operatorname{Ext}}

\def\gldim{\operatorname{gldim}}

\def\Hom{\operatorname{Hom}}

\def\id{\operatorname{id}}

\def\Im{\operatorname{Im}}

\def\pd{{\partial}}

\def\Proj{\operatorname{Proj}}

\def\Sym{\operatorname{Sym}}

\def\uExt{\operatorname{\underline{Ext}}}

\def\ul1{\operatorname{\underline{1}}}

\def\l{\leftarrow}

\def\d{\downarrow}

\def\a{\alpha}
\def\b{\beta}
\def\c{\gamma}
\def\d{\delta}

\def\g{\gamma}
\def\l{\lambda}
\def\om{\omega}
\def\s{\sigma}
\def\t{\tau}

\def\sfv{{\sf v}}
\def\sfw{{\sf w}}

\def\cal{\mathcal}

\def\cD{{\cal D}}

\def\cL{{\cal L}}

\def\cV{{\cal V}}

\def\dirlim{\mathop{\vtop{\baselineskip -100pt\lineskip -1pt\lineskiplimit 0pt
\setbox0\hbox{lim}\copy0\hbox to \wd0{\rightarrowfill}}}\limits}
\def\invlim{\mathop{\vtop{\baselineskip -100pt\lineskip -1pt\lineskiplimit 0pt
\setbox0\hbox{lim}\copy0\hbox to \wd0{\leftarrowfill}}}\limits}

\def\I11{{1 \kern -0.8pt \! \mbox{l}}}
\def\mumu{{\mu\kern-4.2pt\mu}}
\def\bfmu{{\mu\kern-4.2pt\mu}}
\def\2slash{\backslash \! \backslash}

\def\boxtimes{\setbox0\hbox{$\Box$}\copy0\kern-\wd0\hbox{$\times$}}

\def\hdet{\operatorname{hdet}}

\def\GrAut{\operatorname {GrAut}}

\def\<{\langle}
\def\>{\rangle}

\usepackage{geometry}
\numberwithin{equation}{section}

\def\red{\operatorname{red}}

\begin{document}
 
\title{The Classification of 3-dimensional Noetherian Cubic Calabi-Yau Algebras}

\author{Izuru Mori}

\address{
Department of Mathematics,
Graduate School of Science,
Shizuoka University,
836 Ohya, Suruga-ku, Shizuoka 422-8529, Japan}
\email{mori.izuru@shizuoka.ac.jp}

\author{Kenta Ueyama}

\address{
Department of Mathematics,
Faculty of Education,
Hirosaki University,
1 Bunkyocho, Hirosaki, Aomori 036-8560, Japan}
\email{k-ueyama@hirosaki-u.ac.jp}

\thanks {
The first author was supported by JSPS Grant-in-Aid for Scientific Research (C) 25400037.
The second author was supported by JSPS Grant-in-Aid for Young Scientists (B) 15K17503.} 




\begin{abstract} It is known that every 3-dimensional noetherian Calabi-Yau algebra generated in degree 1 is isomorphic to a Jacobian algebra of a superpotential.  Recently, S. P. Smith and the first author classified all superpotentials whose Jacobian algebras are 3-dimensional noetherian quadratic Calabi-Yau algebras.  The main result of this paper is to classify all superpotentials whose Jacobian algebras are 3-dimensional noetherian cubic Calabi-Yau algebras.  As an application, we show that if $S$ is a 3-dimensional noetherian cubic Calabi-Yau algebra and $\sigma$ is a graded algebra automorphism of $S$, then the homological determinant of $\sigma$ can be calculated by the formula $\operatorname{hdet} \sigma=(\operatorname{det} \sigma)^2$ with one exception.
\end{abstract}

\maketitle

\section{Introduction} 

Throughout this paper, we fix an algebraically closed field $k$ of characteristic 0, and we assume that all vector spaces, algebras and unadorned tensor products are over $k$.  
In noncommutative algebraic geometry, AS-regular algebras are the most important class of algebras to study.  In fact, the classification of 3-dimensional (noetherian) AS-regular algebras (generated in degree 1) using algebraic geometry is regarded as a starting point of noncommutative algebraic geometry (\cite {AS}, \cite{ATV}).  On the other hand, in representation theory of algebras, Calabi-Yau algebras are important class of algebras to study.   Since every connected graded Calabi-Yau algebra is AS-regular (\cite {RRZ}), it is interesting to study such algebras from the point of view of both noncommutative algebraic geometry and representation theory.  In particular, since every $m$-Koszul 3-dimensional Calabi-Yau algebra $S$ is isomorphic to a Jacobian algebra $J(\sfw_S)$ of a unique superpotential $\sfw_S$ up to non-zero scalar multiples by \cite {B} (this holds true in the higher dimensional cases by \cite {BSW}, \cite {MS2}), 
it is interesting to study such algebras using both algebraic geometry and superpotentials.  
If $S$ is a 3-dimensional noetherian Calabi-Yau algebra generated in degree 1 over $k$, then $S$ is either 2-Koszul (quadratic) or 3-Koszul (cubic), so $S=J(\sfw_S)$ for some unique superpotential $\sfw_S$.  In \cite {MS1}, S. P. Smith and the first author of this paper focused on studying 3-dimensional noetherian quadratic Calabi-Yau algebras by these points of view.  As a continuation, in this paper, we focus on studying 3-dimensional noetherian cubic Calabi-Yau algebras. 

This paper is organized as follows:  In Section 2, we collect some preliminary results which are needed in this paper.  

Section 3 is the heart of this paper.  Let $S$ be a 3-dimensional noetherian Calabi-Yau algebra.  If $S$ is quadratic, then $\sfw_S\in V^{\otimes 3}$ where $V$ is a 3-dimensional vector space over $k$.  In \cite {MS1},  all $\sfw\in V^{\otimes 3}$ such that $J(\sfw)$ are 3-dimensional Calabi-Yau were classified.  On the other hand, if $S$ is cubic, then $\sfw_S\in V^{\otimes 4}$ where $V$ is a 2-dimensional vector space over $k$.  
A natural next project is to classify all $\sfw\in V^{\otimes 4}$ such that $J(\sfw)$ are 3-dimensional Calabi-Yau.  The main result of this paper is to classify all such superpotentials.  By this complete classification, we obtain the following results: 
\begin{enumerate}
\item{} We compute all possible point schemes for 3-dimensional noetherian cubic Calabi-Yau algebras (Theorem \ref{thm.csp}).  By this computation, we see that not all bidegree (2, 2) divisors in $\PP^1\times \PP^1$ appear as point schemes.  This result contrasts to the fact that all degree 3 divisors in $\PP^2$ appear as point schemes of 3-dimensional noetherian quadratic Calabi-Yau algebras (\cite {MS1}).
\item{} We show that $J(\sfw)$ is 3-dimensional Calabi-Yau except for five algebras up to isomorphisms (Theorem \ref{thm.deg}).  
\item{} We show that $J(\sfw)$ is 3-dimensional Calabi-Yau if and only if it is a domain (Corollary \ref{cor.dom}) as in the quadratic case (\cite {MS1}).    
\end{enumerate}

Section 4 provides a further application of the classification.  The homological determinant plays an important role in invariant theory for AS-regular algebras (\cite{JZ}, \cite{KKZ1}, \cite{KKZ2}, \cite{CKWZ}, \cite{MU} etc.).  Contrary to its importance, it is rather mysterious and not easy to calculate from the definition.  
If $S=T(V)/(R)$ is a 3-dimensional noetherian quadratic Calabi-Yau algebra where $R\subset V\otimes V$, then it was shown in \cite {MS2} that $\hdet \s=\det \s|_V$ for every $\s\in \GrAut S$ if and only if $\sfw_S\not \in \Sym ^3V$.  This means that for most of 3-dimensional noetherian quadratic Calabi-Yau algebras $S=T(V)/(R)$, $\hdet \s=\det \s|_V$ for every $\s\in \GrAut S$.  By \cite[Lemma 7.3 (2)]{KK}, for typical examples of 3-dimensional noetherian cubic AS-regular algebras $S=T(V)/(R)$ where $R\subset V\otimes V\otimes V$, it holds that $\hdet \s=(\det \s|_V)^2$ for every $\s\in \GrAut S$, so it is natural to expect that, for a 3-dimensional noetherian cubic Calabi-Yau algebra $S=T(V)/(R)$, 
$\hdet \s=(\det \s|_V)^2$ for every $\s\in \GrAut S$ if and only if $\sfw_S\not \in \Sym ^4V$.  Although one direction of the above expectation is true (Corollary \ref{cor.hdd}), we will show a slightly more striking result, namely, for a 3-dimensional noetherian cubic Calabi-Yau algebra $S=T(V)/(R)$, $\hdet \s=(\det \s|_V)^2$ for every $\s\in \GrAut S$ if and only if 
$$S\not \cong k\<x, y\>/(xy^2+yxy+y^2x+\sqrt{-3}x^3, yx^2+xyx+x^2y+\sqrt {-3}y^3)$$
(Theorem \ref{thm.exc}).  

For a 3-dimensional noetherian quadratic Calabi-Yau algebra $S=T(V)/(R)$, it was also shown in \cite {MS1} that $\sfw_S\not \in \Sym ^3V$ if and only if $S$ is a deformation quantization of $k[x, y, z]$, and $\sfw_S\in \Sym ^3V$ if and only if $S$ is a Clifford algebra.  In the last section, we will show that something similar holds in one direction (Theorem \ref{thm.dq}, Theorem \ref{thm.cli}).

\section{Preliminaries} 

Throughout this paper, let $k$ be an algebraically closed field of characteristic 0, and $V$ a finite dimensional vector space over $k$.  In this section,  we collect some preliminary results which are needed in this paper.  

\subsection{Superpotentials} 

We define the action of $\theta \in \mathfrak S_m$ on $V^{\otimes m}$ by
$$\theta(v_1\otimes \cdots \otimes v_m):=v_{\theta(1)}\otimes \cdots \otimes v_{\theta(m)}.$$  
Specializing to the $m$-cycle $\phi\in \mathfrak S_m$, we define
$$\phi(v_1\otimes v_2\otimes \cdots \otimes v_{m-1}\otimes v_m):=v_{m}\otimes v_1\otimes \cdots \otimes v_{m-2}\otimes v_{m-1}.$$
We define linear maps  
$c, s, a:V^{\otimes m}\to V^{\otimes m}$ by
\begin{align*}
& c(\sfw):=\frac{1}{m}\sum _{i=0}^{m-1}\phi ^i(\sfw),\\
& s(\sfw):=\frac{1}{|\mathfrak S_m|}\sum _{\theta \in \mathfrak S_m}\theta (\sfw),\\ 
& a(\sfw):=\frac{1}{|\mathfrak S_m|}\sum _{\theta \in \mathfrak S_m}(\operatorname{sgn}\theta) \theta(\sfw).
\end{align*}   
We define the following subspaces of $V^{\otimes m}$: 
\begin{align*}
& \operatorname{Sym}^mV:=\{\sfw \in V^{\otimes m}\mid \theta(\sfw )=\sfw\;  \textnormal { for all } \theta\in \mathfrak S_m\}, \\
& \operatorname{Alt}^mV:=\{\sfw \in V^{\otimes m}\mid \theta(\sfw )=(\operatorname {sgn} \theta)\sfw\;  \textnormal { for all } \theta\in \mathfrak S_m\}.
\end{align*} 
It is easy to see that $\Sym^mV=\Im s$ and $\Alt^mV=\Im a$.  In this paper, $\Im c$ also plays an important role.  
 
\begin{definition} 
\label{de.tw.spp}
Let $\sfw\in V^{\otimes m}$ and $\s \in \GL(V)$. We call $\sfw$
\begin{enumerate}
  \item 
  a \emph{superpotential} if $\phi (\sfw)=\sfw$, 
  \item 
  a \emph{$\s$-twisted superpotential} if $(\s\otimes \id^{\otimes m-1})\phi (\sfw)=\sfw$,
  \item 
  a \emph{twisted superpotential} if it is a  $\s$-twisted superpotential for some $\s$. 
\end{enumerate}  
\end{definition} 

\begin{lemma} \cite [Lemma 6.1]{MS2} 
For $\sfw\in V^{\otimes m}$, $\sfw$ is a superpotential if and only if $\sfw\in \Im c$.  
\end{lemma} 

We call $\Im c$ the \emph{space of superpotentials}. 
For $\s\in \GL(V)$ and $\sfw\in V^{\otimes m}$, we often write $\s(\sfw):=\s^{\otimes m}(\sfw)$ by abuse of notation. 

\begin{lemma} \label{lem.tso} 
For $\theta \in \mathfrak S_m, \s\in \GL(V), \sfw\in V^{\otimes m}$, we have
\begin{enumerate}
\item{} $\theta (\s(\sfw))=\s(\theta(\sfw))$,
\item{} $c(\s(\sfw))=\s(c(\sfw))$,
\item{} $s(\s(\sfw))=\s(s(\sfw))$, 
\item{} $a(\s(\sfw))=\s(a(\sfw))$. 
\end{enumerate}
\end{lemma} 

\begin{proof} For $v_1\otimes \cdots \otimes v_m\in V^{\otimes m}$,
\begin{align*}
\theta (\s(v_1\otimes \cdots \otimes v_m)) 
&=\theta (\s(v_1)\otimes \cdots \otimes \s(v_m)) 
 =\s(v_{\theta (1)})\otimes \cdots \otimes \s(v_{\theta(m)}) \\
& =\s(v_{\theta (1)}\otimes \cdots \otimes v_{\theta(m)}) 
 =\s(\theta(v_1\otimes \cdots \otimes v_m)),
\end{align*}
so $\theta (\s(\sfw))=\s(\theta(\sfw))$.  
It follows that 
\begin{align*}
c(\s(\sfw)) 
 &= \frac{1}{m}\sum _{i=0}^{m-1}\phi^i(\s(\sfw))= \frac{1}{m}\sum _{i=0}^{m-1}\s(\phi^i (\sfw)) 
 = \s\left(\frac{1}{m}\sum _{i=0}^{m-1}\phi^i(\sfw)\right )= \s(c(\sfw)).
\end{align*}
Similarly, we have $s(\s(\sfw))= \s(s(\sfw))$ and $a(\s(\sfw)) = \s(a(\sfw))$.
\end{proof} 

Let $V$ be a finite dimensional vector space over $k$.  The \emph{tensor algebra} of $V$ over $k$ is denoted by $T(V)$, which is an $\NN$-graded algebra by $T(V)_i=V^{\otimes i}$.  An \emph{$m$-homogeneous algebra} is an $\NN$-graded algebra of the form $T(V)/(R)$ where $(R)$ is the two-sided ideal generated by a subspace $R\subset V^{\otimes m}$.  The \emph{symmetric algebra} of $V$ over $k$ is denoted by $S(V)=T(V)/(R)$ where $R=\{u\otimes v-v\otimes u\in V\otimes V\mid u, v\in V\}$, which is an example of a 2-homogeneous algebra (a quadratic algebra).  Note that $S(V)_m=V^{\otimes m}/\sum _{i+j+2=m}V^i\otimes R\otimes V^j$ is the quotient space. 
We denote the quotient map by $\overline {(-)}:V^{\otimes m}\to S(V)_m$.  Since $s(\sfw)=0$ for every $\sfw \in V^i\otimes R\otimes V^j$, the linear map $s:V^{\otimes m}\to V^{\otimes m}$ induces a linear map $\widetilde {(-)}:S(V)_m\to V^{\otimes m}$, 
called the \emph{symmetrization map}.

\begin{lemma} \label{lem.stilde} 
For $\sfw\in V^{\otimes m}$, $\widetilde {\overline \sfw}=s(\sfw)$ holds.
In particular, the linear maps $\overline {(-)}:V^{\otimes m}\to S(V)_m$ and $\widetilde {(-)}:S(V)_m\to V^{\otimes m}$ induce isomorphisms $\overline {(-)}:\Sym^mV\to S(V)_m$ and $\widetilde {(-)}:S(V)_m\to \Sym^mV$ which are inverses to each other. 
\end{lemma}

Note that $\s\in \GL(V)$ extends to $\s\in \GrAut T(V)$ which induces $\s\in \GrAut S(V)$ (by abuse of notations).  

\begin{lemma} \label{lem.qq2} Let $\s\in \GL(V)$. 
\begin{enumerate}
\item{} For $\sfw\in V^{\otimes m}$, $\s(\overline {\sfw})=\overline {\s(\sfw)}$. 
\item{} For $f\in S(V)_m$, $\s\left(\widetilde {f}\right)=\widetilde {\s(f)}$.
\end{enumerate}  
\end{lemma} 

\begin{proof}
(1) This is the way to define $\s\in \GrAut S(V)$ from $\s\in \GrAut T(V)$.  

(2) For every $f\in S(V)_m$, there exists $\sfw\in V^{\otimes m}$ such that $f=\overline {\sfw}$, so 
$$\s\left(\widetilde {f}\right)= \s\left (\widetilde {\overline {\sfw}}\right)=\s(s(\sfw))=s(\s(\sfw))=\widetilde {\overline {\s(\sfw)}}=\widetilde {\s(\overline {\sfw})}
= \widetilde {\s(f)}$$
by Lemma \ref{lem.tso}, Lemma \ref{lem.stilde} and (1). 
\end{proof} 

\subsection{Calabi-Yau Algebras} 

Let $V$ be a vector space, $W\subset V^{\otimes m}$ a subspace and $\sfw\in V^{\otimes m}$.  We introduce the following notation:
\begin{align*}
& \partial W:=\{(\psi \otimes \id ^{\otimes {m-1}})(\sfw)\mid \psi \in V^*, \sfw\in W\}, \\
& W\partial :=\{(\id ^{\otimes {m-1}} \otimes \psi)(\sfw)\mid \psi \in V^*, \sfw\in W\}, \\
& \cD(W, i):=T(V)/(\partial^i W), \\
& \cD(\sfw):=\cD(k\sfw, 1),\\
& J(\sfw):=\cD(c(\sfw)).
\end{align*}
We call $J(\sfw)$ the \emph{Jacobian algebra} of $\sfw$, and $\sfw$ the \emph{potential} of $J(\sfw)$.  
Note that $\cD(\sfw)$ and $J(\sfw)$ are $(m-1)$ homogeneous algebras. 

Choose a basis $x_1, \dots , x_n$ for $V$ so that $T(V)=k\<x_1, \dots, x_n\>$ and $S(V)=k[x_1, \dots, x_n]$.  
For $f\in k[x_1, \dots, x_n]$, the usual partial derivative with respect to $x_i$ is denoted by $f_{x_i}$. 
For a monomial $\sfw =x_{i_1}x_{i_2}\cdots x_{i_{m-1}}x_{i_m}\in k\<x_1, \dots, x_n\>_m$ of degree $m$, we define 
\begin{align*}
& \partial_{x_i}\sfw :=\begin{cases} x_{i_2}\cdots x_{i_{m-1}}x_{i_m} & \textnormal { if } i_1=i \\ 0 & \textnormal { if } i_1\neq i, \end{cases} \quad \textrm{and} \quad
 \sfw \partial_{x_i}:=\begin{cases} x_{i_1}x_{i_2}\cdots x_{i_{m-1}} & \textnormal { if } i_m=i \\ 0 & \textnormal { if } i_m\neq i. \end{cases} 
\end{align*}
We extend the maps $\partial _{x_i}:k\<x_1, \dots, x_n\>_m\to k\<x_1, \dots, x_n\>_{m-1}$ by linearity.  In this notation, $\cD(\sfw)=k\<x_1, \dots, x_n\>/(\partial _{x_1}\sfw, \dots, \partial _{x_n}\sfw)$.  

\begin{lemma} \label{lem.pd} Let $V$ be a vector space with a basis $x_1, \dots, x_n$.  
\begin{enumerate}
\item{} For every $f\in S(V)_m$, $\widetilde {f_{x_i}}=m\partial _{x_i}\widetilde f$. 
\item{} For every $\sfw\in \Sym^mV$, $\widetilde {\overline {\sfw}_{x_i}}=m\partial _{x_i}\sfw$.  
\end{enumerate}
\end{lemma} 

\begin{proof} This follows from \cite [Lemma 2.5]{MS1}. 
\end{proof} 

For an $\NN$-graded algebra $A$, we denote by $A^o$ the opposite graded algebra of $A$, and $A^e=A^o\otimes A$ the enveloping algebra of $A$.   For graded left $A$-modules $M, N$, we denote by $\Ext_A^i(M, N)$ the $i$-th derived functor of $\Hom_A(M, N)$ in the category of graded left $A$-modules, and $\uExt_A^i(M, N):=\bigoplus _{j\in \ZZ}\Ext^i_A(M, N(j))$ where $N(j)=N$ as a left $A$-module with the new grading $N(j)_n=N_{j+n}$.  A \emph{connected graded algebra} is an $\NN$-graded algebra such that $A_0=k$.  In this case, we view $k=A/A_{\geq 1}$ as a graded $A$-module.  

\begin{definition} A connected graded algebra $S$ is called a \emph{$d$-dimensional AS-regular algebra} if 
\begin{enumerate}
\item{} $\gldim S=d<\infty$, and 
\item{} for some $\ell\in \ZZ$, $\uExt^i_S(k, S)\cong \begin{cases} k(\ell) & \textnormal {if } i=d, \\
0 & \textnormal {otherwise.} \end{cases}$
\end{enumerate}
The integer $\ell$ is called the \emph{Gorenstein parameter} of $S$.  
\end{definition} 

\begin{definition} An $\NN$-graded algebra $S$ is called a \emph{$d$-dimensional Calabi-Yau algebra} if 
\begin{enumerate}
\item{} $S$ has a resolution of finite length by finitely generated graded projective left $S^e$-modules, and 
\item{} for some $\ell\in \ZZ$, $\uExt^i_{S^e}(S, S^e)\cong \begin{cases} S(\ell) & \textnormal {if } i=d, \\
0 & \textnormal {otherwise} \end{cases}$ as graded right $S^e$-modules.  
\end{enumerate}
\end{definition} 

It is known that a connected graded $d$-dimensional Calabi-Yau algebra is a $d$-dimensional AS-regular algebra (\cite[Lemma 1.2]{RRZ}). 

\begin{theorem} \cite [Proposition 2.12]{MS2}  If $S$ is a 3-dimensional noetherian AS-regular algebra generated in degree 1 over $k$, then there exists a unique twisted superpotential $\sfw_S$ up to non-zero scalar multiples such that $S=\cD(\sfw_S)$.  
Moreover, if $S$ is Calabi-Yau, then there exists a unique superpotential $\sfw_S$ up to non-zero scalar multiples such that $S=J(\sfw_S)$.  
\end{theorem} 

\begin{lemma} \label{lem.sww'} 
Let $\s \in \GL(V)$ and $W\subset V^{\otimes m}$ a subspace.  If $W'=\s^{\otimes m}(W)\subset V^{\otimes m}$, then  $\s$ extends to an isomorphism of graded algebras $\cD(W, i)\to \cD(W', i)$ for every $i\in \NN$.
\end{lemma}

\begin{proof}
Note that 
\begin{align*}
\s^{\otimes m-1}(\partial W) 
& =\{\s^{\otimes m-1}(\psi \otimes \id^{\otimes m-1})(\sfw)\; | \; \psi\in V^*, \sfw\in W\}
 =\{(\psi \otimes \s^{\otimes m-1})(\sfw )\; | \; \psi \in V^*, \sfw\in W\} \\
& =\{(\psi\s \otimes \s^{\otimes m-1})(\sfw)\; | \; \psi \in V^*, \sfw\in W\} 
 =\{(\psi\otimes \id^{\otimes m-1})\s^{\otimes m}(\sfw)\; | \; \psi \in V^*, \sfw\in W\} \\
& =\partial (\s^{\otimes m}(W)).
\end{align*}
Since $\s^{\otimes m}(W)=W'$, $\s^{\otimes m-i}(\partial ^i W)=\partial ^i(\s^{\otimes m}(W))=\pd^iW'$ for every $i\in \NN$ by induction,  
so $\s$ extends to an isomorphism of graded algebras $\cD(W, i)=T(V)/(\pd^i W)\to \cD(W', i)=T(V)/(\pd^i W')$. 
\end{proof}

A (super)potential $\sfw\in V^{\otimes m}$ is called \emph{Calabi-Yau} if $J(\sfw)$ is Calabi-Yau.  Two potentials $\sfw, \sfw'\in V^{\otimes m}$ are called \emph{equivalent}, denoted by $\sfw \sim \sfw'$, if $\sfw'=\s(\sfw)$ for some $\s\in \GL(V)$.  

\begin{theorem} \label{thm.ws} 
Let $S=T(V)/(R), S'=T(V)/(R')$ be $m$-Koszul $d$-dimensional AS-regular algebras with the Gorenstein parameter $\ell$.  Then $S\cong S'$ as graded algebras if and only if $\sfw_{S}\sim \sfw_{S'}$.  In particular, let $\sfw, \sfw'\in \Im c\subset V^{\otimes m}$ be Calabi-Yau superpotentials.  Then $J(\sfw)\cong J(\sfw')$ as graded algebras if and only if $\sfw\sim \sfw'$.  
\end{theorem}
 
\begin{proof}
An isomorphism of graded algebras $\phi:S=T(V)/(R)\to S'=T(V)/(R')$ restricts to an isomorphism of vector spaces $V^{\otimes m}/R\to V^{\otimes m}/R'$, so $\s:=\phi|_V\in \GL(V)$ has the property that $\s(R)=R'$.  
By \cite[Proposition 2.12]{MS2}, 
\begin{align*}
\s(k\sfw_S) 
 =\s\left(\bigcap _{i+j+m=\ell}V^{\otimes s}\otimes R\otimes V^{\otimes t}\right )
 =\bigcap _{i+j+m=\ell}V^{\otimes s}\otimes R'\otimes V^{\otimes t} 
 =k\sfw_{S'},
\end{align*}
so 
$\sfw_{S'}=\s(\sfw_S)$ by adjusting the scalar.  

Conversely, if $\sfw_{S'}=\s(\sfw_S)$ for some $\s\in \GL(V)$, then $\s(k\sfw_S)=k\sfw_{S'}$, so $\s$ extends to an isomorphism of graded algebras $S=\cD(k\sfw_S, \ell-m)\to S'=\cD(k\sfw_{S'}, \ell-m)$ by \cite [Theorem 1.9]{MS2} and Lemma \ref{lem.sww'}.
\end{proof}

In \cite{ATV}, Artin, Tate and Van den Bergh classified 3-dimensional noetherian AS-regular algebras generated in degree 1 using algebraic geometry.
They showed that every 3-dimensional noetherian quadratic AS-regular algebra determines, and is determined by, a geometric triple $(E, \t, \cL)$ where
$E$ is $\PP^2$ or a divisor in $\PP^2$ of degree 3, $\t$ is an automorphism of $E$, and $\cL$ is an invertible sheaf on $E$,
and also showed that every 3-dimensional noetherian cubic AS-regular algebra determines, and is determined by, a geometric triple $(E, \t, \cL)$ where
$E$ is $\PP^1\times \PP^1$ or a divisor in $\PP^1\times \PP^1$ of bidegree (2, 2), $\t$ is an automorphism of $E$, and $\cL$ is an invertible sheaf on $E$.
The scheme $E$ is called the \emph{point scheme} of $S$, and the pair $(E, \t)$ is called the \emph{geometric pair} of $S$. See \cite{ATV} for details of geometric triples.

\section{Classification of Calabi-Yau Superpotentials} 

In this section, we classify all superpotentials $\sfw$ such that $J(\sfw)$ is 3-dimensional noetherian cubic Calabi-Yau up to isomorphisms of $J(\sfw)$. 

\subsection{Calabi-Yau Properties} 

Let $S=T(V)/(R)$ be an $m$-homogeneous algebra such that $\dim V=\dim R=n$.  If we choose a basis $x_1, \dots, x_n$ for $V$ and a basis $f_1, \dots, f_n$ for $R$, then we may write $f_i=\sum _{j=1}^nm_{ij}x_j$ for some $n\times n$ matrix ${\bf M}=(m_{ij})$ whose entries are in $V^{\otimes {m-1}}$.  Define $g_j=\sum_{i=1}^nx_im_{ij}$.
If we write ${\bf x}=\left(\begin{smallmatrix} x_1 \\ \vdots \\ x_n\end{smallmatrix}\right),
{\bf f}=\left(\begin{smallmatrix} f_1 \\ \vdots \\ f_n \end{smallmatrix}\right),
{\bf g}=\left(\begin{smallmatrix}g_1 \\ \vdots \\ g_n \end{smallmatrix}\right)$, then ${\bf f}={\bf M}{\bf x}$, ${\bf g}=({\bf x}^T{\bf M})^T$.  

\begin{remark} Note that $({\bf x}^T{\bf M})^T\neq {\bf M}^T{\bf x}$. 
\end{remark} 

\begin{definition} \textnormal {\cite {ATV}}
An $m$-homogeneous algebra $S=T(V)/(R)$ is called \emph{standard} if $\dim V=\dim R=n$ and there exist a choice of a basis $x_1, \dots, x_n$ for $V$ and a choice of a basis $f_1, \dots , f_n$ for $R$ such that ${\bf g}=Q_S{\bf f}$ for some $Q_S\in \GL(n, k)$. 
\end{definition} 

Fix a basis $x_1, \dots, x_n$ for $V$.  For $\sfw\in V^{\otimes m}$, we define the $n\times n$ matrix ${\bf M}(\sfw):=(\pd_{x_i}\sfw\pd_{x_j})$ whose entries are in $V^{\otimes m-2}$.
If we write ${\bf x}=\left(\begin{smallmatrix} x_1 \\ \vdots \\ x_n\end{smallmatrix}\right)$, then ${\bf M}(\sfw)$ is the unique matrix such that $\sfw={\bf x}^T{\bf M}(\sfw){\bf x}$.  Viewing ${\bf M}(\sfw)\in M_n(S(V)^{\circ m-2})$, we define the \emph{(noncommutative) Hessian} of $\sfw$ by $H(\sfw)=\det {\bf M}(\sfw)=\det(\partial_{x_i}\sfw \partial _{x_j})\in S(V)^{\circ m-2}$ where $A\circ B$ denotes the Segre product of $A$ and $B$. 

Let $V$ be a vector space with a basis $x_1,  \dots , x_n$ and $\sfw\in V^{\otimes m}$.  
Since $J(\sfw):=T(V)/(\pd _{x_1}c(\sfw), \dots, \pd _{x_n}c(\sfw))$, 
and $\left( \pd _{x_1}c(\sfw) \; \cdots \; \pd _{x_n}c(\sfw) \right)^T ={\bf M}(c(\sfw)){\bf x}$, we have the following results: 

\begin{proposition} \label{prop.ATV}   
Let $V$ be a vector space with a basis $x_1, \dots, x_n$ and $\sfw\in V^{\otimes m}$. 
Then $J(\sfw)$ is standard if and only if $\pd _{x_1}c(\sfw), \dots, \pd_{x_n}c(\sfw)$ are linearly independent. 
\end{proposition} 

\begin{proof} This follows from \cite [Proposition 2.6]{MS1}.  
\end{proof} 

\begin{theorem} \label{thm.atv1}
Let $V$ be a 2-dimensional vector space with a basis $x, y$ 
and $\sfw\in V^{\otimes 4}$.  
Then $J(\sfw)$ is 3-dimensional Calabi-Yau if and only if $\pd _{x}c(\sfw), \pd_{y}c(\sfw)$ are linearly independent and 
$$\cV(\pd_xc(\sfw)\pd _x, \pd_xc(\sfw)\pd _y, \pd_yc(\sfw)\pd _x, \pd_yc(\sfw)\pd _y)
=\emptyset $$ in $\PP^1\times \PP^1$.  If this is the case, then the point scheme of $J(\sfw)$ is given by $\cV(H(c(\sfw)))$.  
\end{theorem} 

\begin{proof} This follows from Proposition \ref{prop.ATV} and \cite [Theorem 1]{ATV}. 
\end{proof} 

\subsection{Classification of four points in $\PP^1$}

Two projective schemes $E, E' \subset \PP(V)$ are called \emph{projectively equivalent}, denoted by $E\sim E'$, if $E'=\s(E)$ for some $\s\in \Aut \PP(V)\cong \PGL(V)$.  Two polynomials $f, f'\in S(V)_m$ are called \emph{equivalent}, denoted by $f\sim f'$, if $f'=\s(f)$ for some $\s\in \GL(V)$.  Note that $f\sim f'$ if and only if
$\Proj S(V)/(f)\sim \Proj S(V)/(f')$.

\begin{lemma} \label{lem.quo} 
Let $V$ be a 2-dimensional vector space with a basis $x, y$.  Every $f\in S(V)_4$ is equivalent to one of the following: 
\begin{enumerate}
\item{} $f_1=0$; 
\item{} $f_2=x^4$;
\item{} $f_3=x^3y$; 
\item{} $f_4=x^2y^2$; 
\item{} $f_5=x^4+x^2y^2$; 
\item{} $g_{\lambda}=x^4+y^4+\l x^2y^2, \l\in k\setminus \{\pm 2\}$. 
\end{enumerate}
Moreover, in case (6), 
$g_{\lambda}\sim g_{\lambda'}$ if and only if $\l' = \pm \l$ or $(2\pm \l)(2\pm \l')= 16$ or $(2\pm \l)(2\mp \l')= 16$.
\end{lemma} 

\begin{proof}   It is equivalent to classify $X:=\Proj k[x, y]/(f)$ up to projective equivalences.  Note that $f\neq 0$ if and only if $X\neq \PP^1$, and, in this case,
$X=\Proj k[x, y]/(f)$ is a set of at most four points.  Note also that every set of three points (with fixed multiplicities) are projectively equivalent to each other.  
If $X=\PP^1$, then $X=\Proj k[x, y]$.
If $X$ consists of one point $(0, 1)$ with multiplicity 4, then $X=\Proj k[x, y]/(x^4)$. 
If $X$ consists of a point $(0, 1)$ with multiplicity 3 and a point $(1, 0)$ of multiplicity 1, then $X=\Proj k[x, y]/(x^3y)$. 
If $X$ consists of two points $(0, 1), (1, 0)$ with multiplicity 2, then $X=\Proj k[x, y]/(x^2y^2)$.
If $X$ consists of a point $(0, 1)$ with multiplicity 2 and two points $(1, \sqrt {-1}), (1, -\sqrt {-1})$ of multiplicity 1, then $X=\Proj k[x, y]/(x^2(x^2+y^2))$. 
If $X=\Proj k[x, y]/(f)$ consists of four points $(0, 1), (1, 0), (1, -1), (1, -\l)$ of multiplicity 1 where $\l\neq 0, 1$, then 
\begin{align*}
f=xy(x+y)(\l x+y)
 \sim x^3y+(\frac{1}{\sqrt \l}+\sqrt \l)x^2y^2+xy^3.  
\end{align*}  
Since $\l \neq 0, 1$, we have $\frac{1}{\sqrt \l}+\sqrt \l\neq \pm 2$, so we may write $f\sim xy(x^2+y^2+\l xy)$ where $\l\neq \pm 2$.  
Moreover, one can check that
\begin{align*}
f & \sim  xy(x^2+y^2+\l xy) 
\sim x^4+y^4-\frac{2\l }{\sqrt {\l^2-4}}x^2y^2.
\end{align*}
It is easy to see that $-\frac{2\l }{\sqrt {\l^2-4}}\neq\pm 2$, so we may write $f\sim x^4+y^4+\l x^2y^2$ where $\l \neq \pm 2$.   On the other hand, if $\l \neq \pm 2$, then it is easy to see that $\Proj k[x, y]/(x^4+y^4+\l x^2y^2)$ consists of four distinct points. 

Write $g_{\lambda}:=x^4+y^4+\l x^2y^2$.  Let $\s\in \GL(2, k)$ such that $\s(g_{\l})=g_{\l'}$ for some $\l , \l '\in k\setminus \{\pm 2\}$.  By a direct calculation, we see that $\s$ is given by
\begin{align*}
\begin{pmatrix} \a & 0 \\ 0 & \d \end{pmatrix}\; (\d^4= \a^4),\quad
\begin{pmatrix} 0 & \b \\ \c & 0 \end{pmatrix}\; (\c^4= \b^4),\;\; \textrm{or}\;\;
\begin{pmatrix} \a & \b \\ -\xi\a & \xi\b \end{pmatrix}\; (\b^4= \a^4, \xi :\textrm{4-th root of unity}). 
\end{align*}
If $\s$ is in the first case or the second case, then we have $\l'=\pm \l$.
If $\s=\left(\begin{smallmatrix} \a & \b \\ -\xi\a & \xi\b \end{smallmatrix}\right)$ with $\b^2= \a^2, \xi^2=1$, then 
\begin{align*}
&(\a x+ \b y)^4+(-\xi\a x + \xi\b y)^4+\l (\a x+ \b y)^2(-\xi\a x + \xi\b y)^2
= \a^4(2+\l )\left(x^4+y^4+\frac{12-2\l}{2+\l}x^2y^2\right),
\end{align*}
so we have $\l' = \frac{12-2\l}{2+\l}$. Thus $(2+\l')(2+\l)= 16$. By similar calculations, we can show that
if $\b^2=-\a^2, \xi^2=1$, then $(2-\l')(2+\l)= 16$, and if $\b^2=\pm \a^2, \xi^2=-1$, then $(2 \pm \l')(2-\l)= 16$,
so the assertion holds.
\end{proof}
 
\subsection{Classification of Calabi-Yau Superpotentials} \label{w_i} 

Let $V$ be a 2-dimensional vector space with a basis $x, y$.  It is clear that the space of superpotentials  $\Im c\subset V^{\otimes 4}$ has a basis 
\begin{align*}
& \sfw_1=x^2y^2+xy^2x+y^2x^2+yx^2y, \\
& \sfw_2=xyxy+yxyx, \\
& \sfw_3=x^3y+x^2yx+xyx^2+yx^3, \\
& \sfw_4=y^3x+y^2xy+yxy^2+xy^3, \\
& \sfw_5=x^4, \\
& \sfw_6=y^4. 
\end{align*}
For the rest of this paper, we fix the above basis for $\Im c$.  It is easy to see that $\Sym^4V$ is a codimension 1 subspace of $\Im c$ spanned by $\sfw_1+\sfw_2, \sfw_3, \sfw_4, \sfw_5, \sfw_6$.  
We set $W':=\bigoplus _{i=3}^6k\sfw_i\subset \Sym^4V$, so that every $\sfw\in \Im c$ can be uniquely written as $\sfw=\a \sfw_1+\b \sfw_2+\sfw'$ where $\a, \b\in k$ and $\sfw'\in W'$.  

Let $\sfw, \sfw'\in \Im c$ be Calabi-Yau superpotentials.  By Theorem \ref{thm.ws} and Lemma \ref{lem.qq2}, if $J(\sfw)\cong J(\sfw')$ as graded algebras, then $\overline \sfw\sim \overline {\sfw'}$,
so we may assume that $\overline {\sfw}\in S(V)_4$ is one of the forms in Lemma \ref{lem.quo} to classify Calabi-Yau superpotentials $\sfw\in \Im c$ up to isomorphisms of $J(\sfw)$.

\begin{theorem} \label{thm.csp}
Let $V$ be a 2-dimensional vector space.  The table below (Table \ref{tab1}) gives a classification of all Calabi-Yau superpotentials $0\neq \sfw\in V^{\otimes 4}$ up to isomorphisms of $J(\sfw)$.  In each Calabi-Yau superpotential $\sfw$, 
the defining relations of $J(\sfw)$, and the geometric pair of $J(\sfw)$ are also listed in the table. 
(In the table, a curve of bidegree $(a, b)$ in $\PP^1\times \PP^1$ is denoted by $(a, b)$ by abuse of notations.  Moreover, we only describe how $\t$ acts on the components of $E$ when $E$ has more than one component.) 
\end{theorem}

{\small
\begin{table}[htbp]
\centering
\scalebox{0.975}[1]{
\begin{tabular}{|l|l|l|l|l|}
\hline
& $\sfw$  &$\partial _x\sfw, \partial _y\sfw$  & $H(\sfw)$  & $(E,\t)$ \\
&(superpotential) &(defining relations of $J(\sfw)$) &(defining equation of $E$) &(geometric pair) \\
\hline \hline
(1)&$\sfw_1-2\sfw_2$ &$xy^2+y^2x-2 yxy$, &$(x_1y_2 - y_1x_2)^2$ &$E=2(1,1)$ \\ 
     & &$yx^2+x^2y-2 xyx$ & &  \\\hline
(2)&$\sfw_1 -2\sfw_2 -2\sfw_5$ &$xy^2+y^2x-2 yxy -2x^3$, &$(x_1y_2 - y_1x_2 -x_1x_2)$ &$E=(1,1)+(1,1)$ \\ 
     & &$yx^2+x^2y-2 xyx$ &\;$\times(x_1y_2 - y_1x_2 +x_1x_2)$ &meeting at 1 pt, \\
     &&&&$\t$ stabilizes \\
     &&&&two components.\\\hline 
(3)&$\sfw_1 -2\sfw_2 +\sfw_3$ &$xy^2+y^2x -2yxy$ &$2(x_1y_2 -y_1x_2)^2$ & $E$ is an irr.~curve\\ 
     & &\;\;$+x^2y+xyx+yx^2$, &\;$+2(y_1x_2 +x_1y_2)x_1x_2$ &with a cusp. \\ 
     &&$yx^2+x^2y -2xyx +x^3$ &\;\;$-x_1^2x_2^2$& \\\hline
(4.1)&$\sfw_1$ &$xy^2+y^2x$, &$0$ & $E=\PP^1\times \PP^1$ \\ 
     & &$yx^2+x^2y$ & & \\ \hline
(4.2)&$\sfw_1+2\sfw_2$ &$xy^2+y^2x+2 yxy$, &$(x_1y_2 + y_1x_2)^2$ &$E=2(1,1)$ \\ 
     & &$yx^2+x^2y+2 xyx$ & & \\\hline
(4.3)&$\a\sfw_1+\b\sfw_2$ &$\a(xy^2+y^2x)+\b yxy$, &$\a x_1^2y_2^2+\a y_1^2x_2^2$ &$E=(1,1)+(1,1)$ \\
     &($\a \neq0, \b \neq 0,$ &$\a(yx^2+x^2y)+\b xyx$ &\;$+\b x_1y_1x_2y_2$ &meeting at 2 pts, \\ 
     &$\b \neq \pm 2\a$)&&&$\t$ stabilizes  \\
     &&&&two components.\\\hline
(5.1)&$\sfw_1+4\sfw_5$ &$xy^2+y^2x +4x^3$, &$x_1^2x_2^2$ &$E=2(1,0)+2(0,1)$, \\
     & &$yx^2+x^2y$ & &$\t$ interchanges  \\
     &&&&two components\\\hline
(5.2)&$\sfw_1+2\sfw_2+8\sfw_5$ &$xy^2+y^2x+2 yxy +8x^3$, &$(x_1y_2 + y_1x_2 -2x_1x_2)$ &$E=(1,1)+(1,1)$ \\ 
     & &$yx^2+x^2y+2 xyx$ &\;$\times(x_1y_2 + y_1x_2 +2x_1x_2)$ &meeting at 1 pt, \\
     &&&&$\t$ interchanges \\
     &&&&two components.\\\hline   
(5.3)&$\a \sfw_1+\b\sfw_2+\sfw_5$ &$\a(xy^2+y^2x)+\b yxy +x^3$, &$\a\b x_1^2y_2^2+\a\b y_1^2x_2^2 $ & $E$ is an irr.~curve\\ 
     &($\a\neq 0,\b \neq 0,$ &$\a(yx^2+x^2y)+\b xyx$ &\;$+\b^2x_1y_1x_2y_2 -\a x_1^2x_2^2$ &with a biflecnode. \\ 
     &$\b \neq \pm 2\a$)&&& \\\hline
(6.1)&$\a \sfw_1+\b\sfw_2+\sfw_5+\sfw_6$ &$\c yxy +x^3$, &$x_1y_1x_2y_2$ &$E=(1,0)+(1,0)$ \\
     &($\a =0, \b \neq 0,\pm 1$) or  &$\c xyx +y^3$ & &\qquad $+(0,1)+(0,1)$, \\
     &($\b = \pm 1, \a \neq 0,\pm 1$)&($\c \neq 0, \pm 1$)&&$\t$ circulates\\
     &&&&four components.\\\hline
(6.2)&$\a \sfw_1+\b\sfw_2+\sfw_5+\sfw_6$ &$xy^2+y^2x +\c x^3$, &$x_1^2x_2^2 +y_1^2y_2^2$ &$E=(1,1)+(1, 1)$  \\
     &($\b =0, \a \neq 0,\pm \frac{1}{2}$) or &$yx^2+x^2y +\c y^3$ &\;$+\c x_1y_1x_2y_2$ &meeting at 2 pts, \\
     &($\b=2\a\pm1, \b \neq 0,\pm 1$) &($\c \neq 0, \pm 2$)&&$\t$ interchanges  \\
     &&&&two components.\\\hline
(6.3)&$\a \sfw_1+\b\sfw_2+\sfw_5+\sfw_6$ &$\a(xy^2+y^2x)+\b yxy +x^3$, &$\a\b x_1^2y_2^2 +\a\b y_1^2x_2^2$ &$E$ is a smooth curve. \\
     &($\a \neq 0, \b\neq 0,\pm 1,$ &$\a(yx^2+x^2y)+\b xyx +y^3$ &\;$+(\b^2-1)x_1y_1x_2y_2$ 
     & \\
     &$\b\neq 2\a\pm 1,$ &&\;\;$-\a x_1^2x_2^2 -\a y_1^2y_2^2$ & \\ 
     &$\b\neq -2\a\pm 1)$ &&& \\\hline
\end{tabular}
}
\vspace{2truemm}
\caption{Classification of Calabi-Yau superpotentials}
\label{tab1}
\end{table}
}

\begin{proof} 
Let $V$ be a 2-dimensional vector space with a basis $x, y$ and $\sfw\in \Im c$ a superpotential.  We divide the cases according to the classification of $\overline \sfw\in S(V)_4$ as in Lemma \ref{lem.quo}.  Note that if $(E, \t)$ is a geometric pair of $J(\sfw)$, then we can compute $\t\in \Aut E$ by the condition $\t((p_1, q_1), (p_2, q_2))=((p_2, q_2), (p_3, q_3))$ such that $((p_1, q_1), (p_2, q_2)), ((p_2, q_2), (p_3, q_3))\in E$ by \cite {ATV}.  

\medskip \noindent (Case 1) If $\overline {\sfw}=f_1= 0$, then $\sfw= \a\sfw_1 + \b\sfw_2$ for $\a, \b\in k$ such that $4\a+2\b=0$.
Since $\sfw \neq 0$, we have $\a\neq 0$, so we may assume that $\a=1$.
Since 
\begin{align*}
\partial _x\sfw  =xy^2+y^2x -2yxy, \quad 
\partial _y\sfw  =yx^2+x^2y -2xyx
\end{align*}
are linearly independent, $J(\sfw)$ is standard by Proposition \ref{prop.ATV}.
It is easy to see that ${\bf M}(\sfw)=\left(\begin{smallmatrix} y^2 &  xy-2yx \\ yx-2xy & x^2 \end{smallmatrix}\right)$. 
Moreover we can calculate $\cV(y^2 , xy-2yx, yx-2xy, x^2)=\emptyset$ in  $\PP^1\times \PP^1$.
Thus $J(\sfw)$ is 3-dimensional Calabi-Yau by Theorem \ref{thm.atv1}.
The point scheme $E$ of $J(\sfw)$ is given by 
\begin{align*}
 \cV(H(\sfw))&=\cV\left(\det \begin{pmatrix} y_1y_2 & x_1y_2-2y_1x_2 \\ y_1x_2-2x_1y_2 & x_1x_2 \end{pmatrix}\right) \\
&= \cV( (x_1y_2-y_1x_2)^2 ) \subset \PP^1\times \PP^1
\end{align*}
(a double curve of bidegree (1, 1)) by Theorem \ref{thm.atv1}. 

\medskip \noindent (Case 2) If $\overline {\sfw}=f_2=x^4$, then $\sfw= \a\sfw_1 + \b\sfw_2+ \sfw_5$ for $\a, \b\in k$ such that $4\a+2\b=0$. We have
\begin{align*}
\partial _x\sfw = \a(xy^2+y^2x -2yxy) +x^3, \quad
\partial _y\sfw = \a(yx^2+x^2y -2xyx).
\end{align*}

If $\a=0$, then
\begin{align*}
\partial _x\sfw  = x^3, \quad
\partial _y\sfw  = 0,
\end{align*}
so $J(\sfw)$ is not 3-dimensional Calabi-Yau by Theorem \ref{thm.atv1}.

If $\a\neq 0$, then we may assume that $\a= -\frac{1}{2}$ and so $\sfw= \sfw_1 -2\sfw_2 -2\sfw_5$. Since
\begin{align*}
\partial _x\sfw = xy^2+y^2x -2yxy -2x^3, \quad
\partial _y\sfw = yx^2+x^2y -2xyx.
\end{align*}
are linearly independent, $J(\sfw)$ is standard by Proposition \ref{prop.ATV}.
It is easy to see that ${\bf M}(\sfw)=\left(\begin{smallmatrix} y^2 -2x^2 &  xy-2yx \\ yx-2xy & x^2 \end{smallmatrix}\right)$. 
Moreover we can calculate $\cV(y^2-2x^2 , xy-2yx, yx-2xy, x^2)=\emptyset$ in  $\PP^1\times \PP^1$. Thus $J(\sfw)$ is 3-dimensional Calabi-Yau by Theorem \ref{thm.atv1}.
The point scheme $E$ of $J(\sfw)$ is given by 
\begin{align*}
\cV(H(\sfw))
&=\cV\left(\det \begin{pmatrix} y_1y_2 -2x_1x_2 & x_1y_2-2y_1x_2 \\ y_1x_2-2x_1y_2 & x_1x_2 \end{pmatrix}\right) \\
&= \cV( (x_1y_2-y_1x_2)^2 -x_1^2x_2^2 ) \\
&= \cV( x_1y_2-y_1x_2 -x_1x_2 ) \cup \cV( x_1y_2-y_1x_2 +x_1x_2 )  \subset \PP^1\times \PP^1
\end{align*}
(the union of two curves of bidegree (1, 1) meeting at one point) by Theorem \ref{thm.atv1}.
The automorphism $\t\in \Aut _kE$ is given by
\begin{align*}
&\t|_{\cV(x_1y_2 - y_1x_2 -x_1x_2)}((p_1, q_1), (p_1, p_1+q_1)) = ((p_1, p_1+q_1), (p_1, 2p_1+q_1)),\\
&\t|_{\cV(x_1y_2 - y_1x_2 +x_1x_2)}((p_1, q_1), (p_1, -p_1+q_1)) = ((p_1, -p_1+q_1), (p_1, -2p_1+q_1))
\end{align*}
($\t$ stabilizes two components).

\medskip \noindent (Case 3) If $\overline {\sfw}=4f_3=4x^3y$, then  $\sfw= \a\sfw_1 + \b\sfw_2+ \sfw_3$ for $\a, \b\in k$ such that $4\a+2\b=0$. Since
\begin{align*}
\partial _x\sfw  =\a(xy^2+y^2x -2yxy) + x^2y+xyx+yx^2, \quad
\partial _y\sfw  =\a(yx^2+x^2y -2xyx) + x^3
\end{align*}
are linearly independent, $J(\sfw)$ is standard by Proposition \ref{prop.ATV}.
It is easy to see that ${\bf M}(\sfw)=\left(\begin{smallmatrix} \a xy+yx+y^2 &  \a(xy-2yx)+x^2 \\ \a(yx-2xy)+x^2 & \a x^2 \end{smallmatrix}\right)$. 
By Theorem \ref{thm.atv1}, we can calculate that
\begin{align*}
&\textrm{ $J(\sfw)$ is 3-dimensional Calabi-Yau}\\
&\Longleftrightarrow \cV(\a xy+yx+y^2, \a(xy-2yx)+x^2, \a(yx-2xy)+x^2, \a x^2)=\emptyset \textrm{ in } \PP^1\times \PP^1\\
&\Longleftrightarrow \a\neq 0.  
\end{align*}
Now assume that $\a\neq 0$. Then we may assume that $\a =1$.
The point scheme $E$ of $J(\sfw)$ is given by 
\begin{align*}
\cV(H(\sfw))=&\cV\left(\det \begin{pmatrix} x_1y_2+y_1x_2+y_1y_2 &  x_1y_2-2y_1x_2+x_1x_2 \\ y_1x_2-2x_1y_2+x_1x_2 & x_1x_2 \end{pmatrix}\right) \\
= & \cV( 2(x_1y_2 -y_1x_2)^2 +2(y_1x_2 +x_1y_2)x_1x_2 -x_1^2x_2^2 )    \subset \PP^1\times \PP^1
\end{align*}
(an irreducible curve of bidegree (2, 2) with a cusp) by Theorem \ref{thm.atv1}.

\medskip \noindent (Case 4) If $\overline {\sfw}=f_4=x^2y^2$, then $\sfw=\a \sfw_1+\b\sfw_2$ for $\a, \b\in k$ such that $4\a+2\b=1$.  Since 
\begin{align*}
\partial _x\sfw =\a(xy^2+y^2x)+\b yxy, \quad
\partial _y\sfw =\a(yx^2+x^2y)+\b xyx
\end{align*}
are linearly independent, $J(\sfw)$ is standard by Proposition \ref{prop.ATV}.
It is easy to see that ${\bf M}(\sfw)=\left(\begin{smallmatrix} \a y^2 & \a xy+\b yx \\ \a yx+\b xy & \a x^2 \end{smallmatrix}\right)$. 
By Theorem \ref{thm.atv1}, we can calculate that 
\begin{align*}
&\textrm{ $J(\sfw)$ is 3-dimensional Calabi-Yau}\\
&\Longleftrightarrow \cV(\a y^2, \a xy+\b yx, \a yx+\b xy, \a x^2)=\emptyset \textrm{ in } \PP^1\times \PP^1 \\
&\Longleftrightarrow \a\neq 0.  
\end{align*}
Now assume that $\a\neq 0$. Then the point scheme $E$ of $J(\sfw)$ is 
\begin{align*}
\cV(H(\sfw))=&\cV\left(\det \begin{pmatrix} \a y_1y_2 & \a x_1y_2+\b y_1x_2 \\ \a y_1x_2+\b x_1y_2 & \a x_1x_2 \end{pmatrix}\right) \\
= & \cV(\b^2 x_1y_1x_2y_2+\a\b x_1^2y_2^2+ \a\b y_1^2x_2^2)\subset \PP^1\times \PP^1
\end{align*}
by Theorem \ref{thm.atv1}.

(4.1) If $\b=0$ (in this case $\a=\frac{1}{4}$), then $E=\PP^1\times \PP^1$.

(4.2) If $\b\neq 0, \b^2-4\a^2=0$ (in this case $\b=2\a$, so $\a=\frac{1}{8}, \b=\frac{1}{4}$), then 
$E= \cV((x_1y_2 + y_1x_2)^2)$ 
(a double curve of bidegree (1, 1)).

(4.3) If $\b\neq 0, \b^2-4\a^2\neq0$ (that is, $\b\neq 0, \b \neq \pm 2\a$), then 
$E = \cV(x_1y_2 - \g_1 y_1x_2) \cup \cV(x_1y_2 - \g_2 y_1x_2)$
(the union of two curves of bidegree (1, 1) meeting at two points) where $\g_1 = \frac{-\b+\sqrt{\b^2-4\a^2}}{2\a}, \g_2 = \frac{-\b-\sqrt{\b^2-4\a^2}}{2\a}$,
and the automorphism $\t\in \Aut _kE$ is given by 
\begin{align*}
&\t|_{\cV(x_1y_2 - \g_1 y_1x_2)}((p_1,q_1), (p_1,\g_1 q_1)) = ((p_1,\g_1 q_1),(p_1, \g_1^2q_1)),\\
&\t|_{\cV(x_1y_2 - \g_2 y_1x_2)}((p_1,q_1), (p_1,\g_2 q_1)) = ((p_1,\g_2 q_1),(p_1, \g_2^2q_1))
\end{align*}
($\t$ stabilizes the components).

\medskip \noindent (Case 5) If $\overline {\sfw}=f_5=x^4+x^2y^2$, then $\sfw=\a \sfw_1+\b\sfw_2+\sfw_5$ for $\a, \b\in k$ such that $4\a+2\b=1$.
Since 
\begin{align*}
\partial _x\sfw =\a(xy^2+y^2x)+\b yxy +x^3, \quad
\partial _y\sfw =\a(yx^2+x^2y)+\b xyx
\end{align*}
are linearly independent, $J(\sfw)$ is standard by Proposition \ref{prop.ATV}.
It is easy to see that ${\bf M}(\sfw)=\left(\begin{smallmatrix} \a y^2+x^2 & \a xy+\b yx \\ \a yx+\b xy & \a x^2 \end{smallmatrix}\right)$. 
By Theorem \ref{thm.atv1}, we can calculate that 
\begin{align*}
&\textrm{ $J(\sfw)$ is 3-dimensional Calabi-Yau}\\
&\Longleftrightarrow \cV(\a y^2+x^2, \a xy+\b yx, \a yx+\b xy, \a x^2)=\emptyset \textrm{ in } \PP^1\times \PP^1\\
&\Longleftrightarrow \a\neq 0.  
\end{align*} 
Now assume that $\a\neq 0$. Then the point scheme $E$ of $J(\sfw)$ is 
\begin{align*}
\cV(H(\sfw))=&\cV\left(\det \begin{pmatrix} \a y_1y_2 +x_1x_2 & \a x_1y_2+\b y_1x_2 \\ \a y_1x_2+\b x_1y_2 & \a x_1x_2 \end{pmatrix}\right) \\
= & \cV(\b^2x_1y_1x_2y_2+\a\b x_1^2y_2^2 +\a\b y_1^2x_2^2 -\a x_1^2x_2^2)\subset \PP^1\times \PP^1
\end{align*}
by Theorem \ref{thm.atv1}.

(5.1) If $\b=0$ (in this case $\a=\frac{1}{4}$), then $E= \cV(x_1^2) \cup \cV(x_2^2)$
(the union of a double curve of bidegree (1, 0) and a double curve of bidegree (0, 1)),
and the automorphism $\t\in \Aut _kE_{\red}$ is given by
\begin{align*}
\t|_{\cV(x_1^2)}((0,1), (p_2,q_2)) = ((p_2,q_2), (0,1)),\\
\t|_{\cV(x_2^2)}((p_1,q_1), (0,1)) = ((0,1), (p_1,-q_1))
\end{align*}
($\t$ interchanges two components).

(5.2) If $\b\neq 0$, $\b^2-4\a^2 = 0$  (in this case $\b=2\a$, so $\a=\frac{1}{8}, \b=\frac{1}{4}$),
then 
$$E = \cV(\frac{1}{2}(x_1y_2 + y_1x_2) -x_1x_2) \cup \cV(\frac{1}{2}(x_1y_2 + y_1x_2) +x_1x_2)$$
(the union of two curves of bidegree (1, 1) meeting at one point),
and the automorphism $\t\in \Aut _kE$ is given by
\begin{align*}
&\t|_{\cV(\frac{1}{2}(x_1y_2 + y_1x_2) -x_1x_2)}((p_1, q_1), (p_1, 2p_1-q_1)) = ((p_1, 2p_1-q_1), (p_1, -4p_1+q_1)),\\
&\t|_{\cV(\frac{1}{2}(x_1y_2 + y_1x_2) +x_1x_2)}((p_1, q_1), (p_1, -2p_1-q_1)) = ((p_1, -2p_1-q_1), (p_1, 4p_1+q_1))
\end{align*}
($\t$ interchanges two components). 

(5.3) If $\b\neq 0$, $\b^2-4\a^2 \neq 0$ (that is, $\b\neq 0, \b \neq \pm 2\a$), then 
$$E = \cV(\b(x_1y_2 - \g_1 y_1x_2)(x_1y_2 - \g_2 y_1x_2) - \a x_1^2x_2^2)$$ 
(an irreducible curve of bidegree (2, 2) with a biflecnode)
where $\g_1 = \frac{-\b+\sqrt{\b^2-4\a^2}}{2\a}, \g_2 = \frac{-\b-\sqrt{\b^2-4\a^2}}{2\a}$.

\medskip \noindent (Case 6) If $\overline {\sfw}=g_\l=x^4+y^4+\l x^2y^2$, then $\sfw=\a \sfw_1+\b\sfw_2+\sfw_5+\sfw_6$ for $\a, \b\in k$ such that $4\a+2\b=\l (\neq \pm 2)$.
Since 
\begin{align*}
\partial _x\sfw =\a(xy^2+y^2x)+\b yxy +x^3, \quad
\partial _y\sfw =\a(yx^2+x^2y)+\b xyx +y^3
\end{align*}
are linearly independent, $J(\sfw)$ is standard by Proposition \ref{prop.ATV}.
It is easy to see that ${\bf M}(\sfw)=\left(\begin{smallmatrix} \a y^2+x^2 & \a xy+\b yx \\ \a yx+\b xy & \a x^2 +y^2 \end{smallmatrix}\right)$. 
By Theorem \ref{thm.atv1}, we can calculate that 
\begin{align*}
&\textrm{ $J(\sfw)$ is 3-dimensional Calabi-Yau}\\
&\Longleftrightarrow \cV(\a y^2+x^2, \a xy+\b yx, \a yx+\b xy, \a x^2+y^2)=\emptyset \textrm{ in } \PP^1\times \PP^1 \\
&\Longleftrightarrow \cV\left(
\begin{pmatrix} \a &\b \\ \b &\a \end{pmatrix}\begin{pmatrix} xy \\ yx \end{pmatrix},
\begin{pmatrix} 1 &\a \\ \a &1 \end{pmatrix}\begin{pmatrix} x^2 \\ y^2 \end{pmatrix}
\right) \textrm{ in } \PP^1\times \PP^1 \\
&\Longleftrightarrow \a^2-\b^2 \neq 0,\; \a(\a^2-1) \neq 0, \textrm{ or } \b(\a^2-1) \neq 0\\
&\Longleftrightarrow(\a,\b) \neq (0,0),(\pm 1,\pm 1) \qquad (\textrm{since } \l \neq \pm 2, (\a,\b) \textrm{ can not be } (\pm 1, \mp 1)).
\end{align*}
Now assume that $(\a,\b) \neq (0,0),(\pm 1,\pm 1)$. Then the point scheme $E$ of $J(\sfw)$ is 
\begin{align*}
\cV(H(\sfw))=&\cV\left(\det \begin{pmatrix} \a y_1y_2 +x_1x_2 & \a x_1y_2+\b y_1x_2 \\ \a y_1x_2+\b x_1y_2 & \a x_1x_2 +y_1y_2 \end{pmatrix}\right) \\
= & \cV((\b^2-1)x_1y_1x_2y_2+\a\b x_1^2y_2^2 +\a\b y_1^2x_2^2 -\a x_1^2x_2^2 -\a y_1^2y_2^2)\subset \PP^1\times \PP^1
\end{align*}
by Theorem \ref{thm.atv1}

If $\a=0, \b^2-1=0$, then $\l = 4\a+2\b=\pm2$, so this case is excluded.

(6.1) If $\a=0, \b^2-1\neq 0$, then 
$E= \cV(x_1)\cup\cV(y_1)\cup\cV(x_2)\cup\cV(y_2)$
(the union of two curves of bidegree (1, 0) and two curves of bidegree (0, 1)),
and the automorphism $\t\in \Aut _kE$ is given by
\begin{align*}
&\t|_{\cV(x_1)}((0,1), (p_2,q_2)) = ((p_2,q_2), (1,0)),\\
&\t|_{\cV(y_1)}((1,0), (p_2,q_2)) = ((p_2,q_2), (0,1)),\\
&\t|_{\cV(x_2)}((p_1,q_1), (0,1)) = ((0,1), (q_1,-p_1)),\\
&\t|_{\cV(y_2)}((p_1,q_1), (1,0)) = ((1,0), (q_1,-p_1))
\end{align*}
($\t$ circulates four components). 

If $\a\neq0, \b=0, 4\a^2-1=0$, then $\l = 4\a+2\b=\pm2$, so this case is excluded.

(6.2) If $\a\neq0, \b=0, 4\a^2-1 \neq 0$, then 
$E= \cV(x_1x_2- \g_1 y_1y_2)\cup \cV(x_1x_2- \g_2 y_1y_2)$
(the union of two curves of bidegree (1, 1) meeting at two points)
where $\g_1 = \frac{-1+\sqrt{1-4\a^2}}{2\a}, \g_2 = \frac{-1-\sqrt{1-4\a^2}}{2\a}$,
and the automorphism $\t\in \Aut _kE$ is given by
\begin{align*}
&\t|_{\cV(x_1x_2 - \g_1 y_1y_2)}((p_1,q_1), (\g_1 q_1, p_1)) = ((\g_1 q_1, p_1),(\g_2 p_1, \g_1 q_1)),\\
&\t|_{\cV(x_1x_2 - \g_2 y_1y_2)}((p_1,q_1), (\g_2 q_1, p_1)) = ((\g_1 q_1, p_1),(\g_1 p_1, \g_2 q_1))
\end{align*}
($\t$ interchanges two components).

We now consider the case $\a\neq 0, \b^2-1=0$.
If $\b=-1$, then $J(\sfw)$ is isomorphic to an algebra in the case $\b=1$,
so we may assume that $\b=1$ (in this case $\a \neq \pm1$).  Then 
$J(\sfw)$ is $k\<x,y\>(\a(xy^2+y^2x)+ yxy +x^3, \a(yx^2+x^2y)+ xyx +y^3)$.
The homomorphism defined by $x \mapsto x-y, y \mapsto x+y$ induces the isomorphism 
\[ J(\sfw) \xrightarrow{\sim} k\<x,y\>/(\frac{1-\a}{1+\a}yxy+x^3, \frac{1-\a}{1+\a}xyx+y^3 ). \]
We can check that $\frac{1-\a}{1+\a} \neq 0, \pm 1$, so this case is integrated with the case (6.1).

We next consider the case $\a \neq 0, \b \neq 0, \b^2-1\neq0, 2\a-\b= \pm 1$.
If $2\a-\b=-1$, then $J(\sfw)$ is isomorphic to an algebra in the case $2\a-\b=1$,
so we may assume that $2\a-\b=1$.
Then
$J(\sfw)$ is $k\<x,y\>(\a(xy^2+y^2x)+ (2\a-1)yxy +x^3, \a(yx^2+x^2y)+ (2\a-1)xyx +y^3)$.
The homomorphism defined by $x \mapsto x-y, y \mapsto x+y$ induces the isomorphism 
\[ J(\sfw) \xrightarrow{\sim} k\<x,y\>/(\frac{1-\a}{2\a}(xy^2+y^2x)+x^3, \frac{1-\a}{2\a}(xy^2+y^2x)+y^3 ). \]
We can check that $\frac{1-\a}{2\a} \neq 0, \pm \frac{1}{2}$, so this case is integrated with the case (6.2).

(6.3) If $\a \neq 0, \b \neq 0, \b^2-1\neq0, 2\a-\b\neq \pm 1$, then we can calculate that $E$ is a smooth curve of bidegree (2, 2)
by using the Jacobian criterion.
\end{proof}

\begin{theorem} \label{thm.deg}
Let $V$ be a 2-dimensional vector space and $0\neq \sfw\in V^{\otimes 4}$.   If $J(\sfw)$ is not 3-dimensional Calabi-Yau, then $J(\sfw)$ is isomorphic to one of the following five algebras: 
\begin{enumerate}
\item{} $k\<x, y\>/(x^3)$;
\item{} $k\<x, y\>/(x^3, x^2y+xyx+yx^2)$;
\item{} $k\<x, y\>/(yxy, xyx)$;
\item{} $k\<x, y\>/(yxy+x^3, xyx)$; 
\item{} $k\<x, y\>/(x^3, y^3)$. 
\end{enumerate}
\end{theorem} 

\begin{proof} In each (Case 2), (Case 3), (Case 4), (Case 5) in the proof of Theorem \ref{thm.csp}, non-Calabi-Yau algebras are given by (1), (2), (3), (4) as in the above theorem. 
In (Case 6), non-Calabi-Yau algebras are isomorphic to either $k\<x, y\>/(x^3, y^3)$ or $k\<x, y\>/(xy^2+y^2x+yxy\pm x^3, yx^2+x^2y+xyx\pm y^3)$ by the proof of Theorem \ref{thm.csp}.  For $a\in k$, 
$$(x+ay)^3=a^2(xy^2+y^2x+yxy)+x^3+a(yx^2+x^2y+xyx)+a^3y^3,$$
so we have
\begin{align*}
&  k\<x, y\>/(xy^2+y^2x+yxy+x^3, yx^2+x^2y+xyx+y^3)=k\<x, y\>/((x+y)^3, (x-y)^3), \\
&  k\<x, y\>/(xy^2+y^2x+yxy-x^3, yx^2+x^2y+xyx-y^3)=k\<x, y\>/((x+\sqrt {-1}y)^3, (x-\sqrt {-1}y)^3),
\end{align*}
hence the result. 
\end{proof} 

\begin{corollary} \label {cor.dom} 
Let $V$ be a 2-dimensional vector space and $0\neq \sfw\in V^{\otimes 4}$.  Then $J(\sfw)$ is 3-dimensional Calabi-Yau if and only if it is a domain.  
\end{corollary}  

\begin{proof} If $J(\sfw)$ is 3-dimensional Calabi-Yau, then it is 3-dimensional noetherian AS-regular, so it is a domain by \cite[Theorem 3.9]{ATV2}.  On the other hand, if $J(\sfw)$ is not 3-dimensional Calabi-Yau, then it is not a domain by the classification in Theorem \ref{thm.deg}.
\end{proof}  

\section{Homological Determinants}

The homological determinant plays an important role in invariant theory for AS-regular algebras (\cite{JZ}, \cite{KKZ1}, \cite{KKZ2}, \cite{CKWZ}, \cite{MU} etc.).  For a 3-dimensional noetherian quadratic Calabi-Yau algebra $S=J(\sfw)=T(V)/(R)$, it was shown in \cite [Theorem 7.2]{MS2} that $\hdet \s=\det \s|_V$ for every $\s\in \GrAut S$ if and only if $c(\sfw)\not \in \Sym ^3V$.  In this section, we will compute homological determinants for 3-dimensional noetherian cubic Calabi-Yau algebras. 

\begin{lemma} \label{lem.mhdd}
Let $n=\dim V$, $m\in \NN^+$, 
and $\s\in \GL(V)$.  
For $\sfw\in (\Alt^nV)^{\otimes m}$, $\s(\sfw)=(\det \s|_V)^m\sfw$.
\end{lemma}

\begin{proof} Fix a basis $x_1, \dots, x_n$ for $V$.  First, we show it for $m=1$.  Let $\sfw=x_{i_1}\otimes \cdots \otimes x_{i_n}\in V^{\otimes n}$ be a monomial.  Since $a(\theta (\sfw))=\sgn (\theta )a(\sfw)$ for $\theta\in S_n$, if $x_{i_s}=x_{i_t}$ for some $s\neq t$, then $a(\sfw)=0$, and if $i_s\neq i_t$ for every $s\neq t$, then $a(\sfw)=\pm  a(x_1\otimes \cdots \otimes x_n)$, so $\Alt^nV=\Im a$ is a one dimensional vector space spanned by $\sfw_0:=a(x_1\otimes \cdots \otimes x_n)$.  Since $S=\cD(\sfw_0, n-2)=k[x_1, \dots, x_n]$ is the polynomial algebra, $\s$ extends to a graded algebra automorphism of $S$, so $\s(\sfw_0)=(\hdet \s)\sfw_0$ by \cite [Theorem 3.3]{MS2}, but for a polynomial algebra, it is known that $\hdet \s=\det \s|_V$, hence $\s(\sfw_0)=(\det \s|_V)\sfw_0$.  

For $m>1$, if $\sfw=\sfv_1\otimes \cdots \otimes \sfv_m\in (\Alt^nV)^{\otimes m}$ where $\sfv_i\in \Alt^nV$, then 
$$\s(\sfw)=\s(\sfv_1)\otimes \cdots \otimes \s(\sfv_m)=(\det \s|_V)\sfv_1\otimes \cdots \otimes (\det \s|_V)\sfv_m=(\det \s|_V)^m\sfw.$$  
The result follows by linearity.
\end{proof} 

Let $V$ be a 2-dimensional vector space.  
Since $V^{\otimes 2}=\Sym^2V\oplus \Alt^2V$, 
\begin{align*}
V^{\otimes 4} & =(\Sym^2V\oplus \Alt^2V)\otimes (\Sym^2V\oplus \Alt^2V) \\
& =(\Sym^2V)^{\otimes 2}\oplus (\Sym^2V\otimes \Alt^2V)\oplus (\Alt^2V\otimes \Sym^2V)\oplus (\Alt^2V)^{\otimes 2}.
\end{align*} 
Let $\pi :V^{\otimes 4}\to (\Alt^2V)^{\otimes 2}$ be the projection map with respect to the above decomposition.     
Fix a basis $x, y$ for $V$.  Since $\Sym^2V=kx^2+k(xy+yx)+ky^2$ and $\Alt ^2V=k(xy-yx)$, we have
\begin{align*}
& (\Sym^2V)^{\otimes 2}=kx^4+ky^4+kx^2y^2+ky^2x^2+k(x^3y+x^2yx)+k(xyx^2+yx^3) \\
& \hspace{1in} +k(y^2xy+y^3x)+k(xy^3+yxy^2)+k(xyxy+xy^2x+yx^2y+yxyx),\\
& \Sym^2V\otimes \Alt^2V=k(x^3y-x^2yx)+k(y^2xy-y^3x)+k(xyxy-xy^2x+yx^2y-yxyx), \\ 
& \Alt^2V\otimes \Sym^2V=k(xyx^2-yx^3)+k(xy^3-yxy^2)+k(xyxy+xy^2x-yx^2y-yxyx), \\ 
& (\Alt^2V)^{\otimes 2}=k(xyxy-xy^2x-yx^2y+yxyx).
\end{align*}
Let $\sfw_0=xyxy-xy^2x-yx^2y+yxyx$ so that $(\Alt^2V)^{\otimes 2}=k\sfw_0$, and define a map $\mu :V^{\otimes 4}\to k$ by $\pi (\sfw)=\mu (\sfw)\sfw_0$. 

\begin{remark} Although the map $\mu$ depends on the choice of a basis for $V$, the map $\pi$ is independent of the choice of a basis for $V$.  Since $\mu (\sfw)=0$ if and only if $\pi (\sfw)=0$, whether $\mu (\sfw)=0$ or not is independent of the choice of a basis for $V$.  
\end{remark} 

\begin{theorem} \label{thm.hdd} 
Let $S=T(V)/(R)$ be a 3-dimensional noetherian cubic AS-regular algebra where $R\subset V^{\otimes 3}$.
If $\mu (\sfw_S)\neq 0$, then $\hdet \s=(\det \s|_V)^2$ for every $\s\in \GrAut S$. 
\end{theorem} 

\begin{proof} 
Write $W_s:=(\Sym^2V)^{\otimes 2}\oplus (\Sym^2V\otimes \Alt^2V)\oplus (\Alt^2V\otimes \Sym^2V)$ and $W_a:=(\Alt^2V)^{\otimes 2}$ so that $V^{\otimes 4}=W_s\oplus W_a$.  
For $\sfw\in V^{\otimes 2}$, $\s(s(\sfw))=s(\s(\sfw))$ and $\s(a(\sfw))=a(\s(\sfw))$ by Lemma \ref{lem.tso}, so $\s(W_s)=W_s$ and $\s(W_a)=W_a$.  

We may write $\sfw_S=\sfw_s+\sfw_a$ where $\sfw_a:=\pi (\sfw_S)\in W_a, \sfw_s:=\sfw_S-\pi (\sfw_S)\in W_s$ in a unique way.  By \cite [Theorem 3.3]{MS2} and Lemma \ref{lem.mhdd},
$$(\hdet \s)(\sfw_s+\sfw_a)=(\hdet \s)(\sfw_S)=\s(\sfw_S)=\s(\sfw_s+\sfw_a)=\s(\sfw_s)+(\det \s|_V)^2\sfw_a,$$
so 
$$(\hdet \s-(\det \s|_V)^2)\sfw_a=\s(\sfw_s)-(\hdet \s)\sfw_s=0$$
because the left hand side is in $W_a$ and the right hand side is in $W_s$.  
If $\mu (\sfw_S)\neq 0$, 
then $\sfw_a=\pi (\sfw_S)\neq 0$, so $\hdet \s=(\det \s|_V)^2$. 
\end{proof}  

\begin{example} \label{ex.cuas}
If $S=k\<x, y\>/(x^2y-yx^2, y^2x-xy^2)$, then $S$ is a 3-dimensional noetherian cubic AS-regular algebra (but not a Calabi-Yau algebra) such that 
\begin{align*}
\sfw_S & =xy^2x-x^2y^2+yx^2y-y^2x^2 \\
& =-\frac{1}{2}(xyxy-xy^2x-yx^2y+yxyx)+\frac{1}{2}(xyxy+xy^2x+yx^2y+yxyx)-x^2y^2-y^2x^2 \\
& =-\frac{1}{2}\sfw_0+\sfw'
\end{align*}
where $\sfw'\in (\Sym^2V)^{\otimes 2}$, so $\mu (\sfw_S)=-\frac{1}{2}\neq 0$, hence $\hdet \s=(\det \s|_V)^2$ for every $\s\in \GrAut S$ by Theorem \ref{thm.hdd}.     
In fact, it is easy to calculate $\s(\sfw_S)$ for elementary matrices $\s\in \GL(2, k)$:
$$\begin{array}{|c||c|c|c|}
\hline 
\s & \begin{pmatrix} 0 & 1 \\ 1 & 0 \end{pmatrix} & \begin{pmatrix} \a & 0 \\ 0 & 1 \end{pmatrix} & \begin{pmatrix} 1 & \a \\ 0 & 1 \end{pmatrix}\\
\hline 
\det (\s) & -1 & \a & 1 \\
\hline 
\s(\sfw_S) & \sfw_S & \a^2\sfw_S & \sfw_S \\
\hline 
\hdet (\s) & 1 & \a^2 & 1 \\
\hline
\end{array}$$
Since every $\s\in \GL(2, k)$ is a product of elementary matrices, every $\s$ extends to a graded algebra automorphism $\s\in \GrAut S$ by \cite [Theorem 3.2]{MS2}.  Since both $\det $ and $\hdet$ are group homomorphisms, $\hdet \s=(\det \s|_V)^2$ for every $\s\in \GrAut S$ by the above table. 
\end{example} 

Recall that we set $W':=\bigoplus _{i=3}^6k\sfw_i\subset \Sym^4V$ where $\sfw_i$ are as defined in Section \ref{w_i}. 

\begin{lemma} \label{lem.csym}
Let $V$ be a 2-dimensional vector space.  For a superpotential $\sfw\in \Im c\subset V^{\otimes 4}$, $\sfw\in \Sym^4V$
if and only if $\mu (\sfw)=0$.  
\end{lemma} 

\begin{proof} Every $\sfw\in \Im c$ can be written as $\sfw=\a(\sfw_2-\sfw_1)+\b(\sfw_2+\sfw_1)+\sfw'$ where $\a, \b\in k$ and $\sfw'\in W'$.  
It is easy to see that $\sfw_2+\sfw_1, \sfw '\in (\Sym^2V)^{\otimes 2}$. 
Since $x^2y^2, y^2x^2\in (\Sym^2V)^{\otimes 2}$,
$$\mu (\sfw)=\a \mu (\sfw_2-\sfw_1)=\a\mu (\sfw_0+x^2y^2+y^2x^2)=\a.$$
Since $\sfw=(\b-\a)\sfw_1+(\b+\a)\sfw_2+\sfw'$, $\sfw\in \Sym^4V$ if and only if $\b-\a=\b+\a$ if and only if $\mu (\sfw)=\a=0$.  
\end{proof} 

\begin{corollary} \label{cor.hdd} 
Let $S=T(V)/(R)$ be a 3-dimensional noetherian cubic Calabi-Yau algebra where $R\subset V^{\otimes 3}$.  If $\sfw_S\not \in \Sym^4V$, then $\hdet \s=(\det \s|_V)^2$ for every $\s\in \GrAut S$. 
\end{corollary} 

\begin{proof} 
Since $S$ is Calabi-Yau, $\sfw_S\in \Im c$ is a superpotential, so the result follows from Theorem \ref{thm.hdd} and Lemma \ref{lem.csym}. 
\end{proof} 

\begin{proposition} \label{prop.symc}
Every 3-dimensional noetherian cubic Calabi-Yau algebra $S$ such that $\sfw_S\in \Sym ^4V$ is isomorphic to one of the following algebras: 
\begin{enumerate}
\item{} $k\<x, y\>/(xy^2+yxy+y^2x, yx^2+xyx+x^2y)$;
\item{} $k\<x, y\>/(xy^2+yxy+y^2x+x^3, yx^2+xyx+x^2y)$;
\item{} $k\<x, y\>/(xy^2+yxy+y^2x+ax^3, yx^2+xyx+x^2y+ay^3), \; a\in k\setminus \{0, \pm 1,\pm 3\}$.
\end{enumerate}
\end{proposition} 

\begin{proof} This follows from the proof of Theorem \ref{thm.csp}.
\end{proof} 

\begin{theorem} \label{thm.exc}  
Let $S$ be a 3-dimensional noetherian cubic Calabi-Yau algebra.   Then $\hdet \s=(\det \s|_V)^2$ for every $\s\in \GrAut S$ if and only if  
$$S\not \cong k\<x, y\>/(xy^2+yxy+y^2x+\sqrt {-3}x^3, yx^2+xyx+x^2y+\sqrt {-3}y^3).$$
\end{theorem}

\begin{proof} \label{ex.cuas2}
By Corollary \ref{cor.hdd}, it is enough to check the algebras in Proposition \ref{prop.symc}.  Let $\sfw\in \Im c\subset V^{\otimes 4}$ be a Calabi-Yau superpotential.  Then $\s\in \GL(V)$ extends to $\s\in \GrAut J(\sfw)$ if and only if $\s(\sfw)=\l \sfw$ for some $\l \in k$ by \cite[Theorem 3.2]{MS2}.  In this case, $\l=\hdet \s$ by \cite [Theorem 3.3]{MS2}, and 
$\s(\overline \sfw)=\overline {\s(\sfw)}=\overline {\l \sfw}=\l \overline \sfw$ by Lemma \ref{lem.qq2}.  
 
(Case 1) If $S=k\<x, y\>/(xy^2+yxy+y^2x, yx^2+xyx+x^2y)$, then
$\sfw_S=x^2y^2+xy^2x+y^2x^2+yx^2y+xyxy+yxyx\in \Sym^4V$.  Since $\overline \sfw_S=6x^2y^2$, if $\s=\left(\begin{smallmatrix} \a & \b \\ \c & \d \end{smallmatrix}\right) \in \GL(2, k)$ extends to $\s\in \GrAut S$, then 
$$\s(x^2y^2)=(\a x+\b y)^2(\c x+\d y)^2=\a^2\c^2x^4+\cdots +\b^2\d ^2y^4=\l x^2y^2$$
for some $\l\in k$, so $\a\c=\b\d=0$. Since $\a\d-\b\c\neq 0$, it follows that either $\s=\left(\begin{smallmatrix} \a & 0 \\ 0 & \d \end{smallmatrix}\right)$ or $\s=\left(\begin{smallmatrix} 0 & \b \\ \c & 0 \end{smallmatrix}\right)$.  
If $\s=\left(\begin{smallmatrix} \a & 0 \\ 0 & \d \end{smallmatrix}\right)$, then $\hdet \s=\a^2\d^2$ and $\det \s|_V=\a\d$.  
If $\s=\left(\begin{smallmatrix} 0 & \b \\ \c & 0 \end{smallmatrix}\right)$, then $\hdet \s=\b^2\c^2$ and $\det \s|_V=-\b\c$.  
In either case, $\hdet \s=(\det \s|_V)^2$.

(Case 2) If $S=k\<x, y\>/(xy^2+yxy+y^2x+x^3, yx^2+xyx+x^2y)$, then
$\sfw_S=x^2y^2+xy^2x+y^2x^2+yx^2y+xyxy+yxyx + x^4\in \Sym^4V$.
Since $\overline \sfw_S=6x^2y^2+x^4$, if $\s=\left(\begin{smallmatrix} \a & \b \\ \c & \d \end{smallmatrix}\right) \in \GL(2, k)$ extends to $\s\in \GrAut S$, then
$\s(6x^2y^2+x^4)=\l (6x^2y^2+x^4)$ for some $\l\in k$.
Since
\begin{align*}
\s(6x^2y^2+x^4)
&= 6(\a x+\b y)^2(\c x+\d y)^2+(\a x+\b y)^4\\
&= \a^2(6\c^2+\a^2)x^4
+4\a(3\a\c\d +3\b\c^2 +\a^2\b)x^3y
+6(\a^2\d^2 + 4\a\b\c\d +\b^2\c^2 +\a^2\b^2)x^2y^2\\
&\qquad +4\b(3\b\c\d+3\a\d^2+\a\b^2)xy^3
+\b^2(6\d^2+\b^2)y^4,
\end{align*}
we obtain
$4\a(3\a\c\d +3\b\c^2 +\a^2\b)=0$, $4\b(3\b\c\d+3\a\d^2+\a\b^2)=0$, $\b^2(6\d^2+\b^2)=0$, and
$\a^2\d^2 + 4\a\b\c\d +\b^2\c^2 +\a^2\b^2 = \a^2(6\c^2+\a^2)$.
If $\b\neq 0$, then $\b^2=-6\d^2$, so
$0=4\b(3\b\c\d+3\a\d^2+\a\b^2)=4\b(3\b\c\d-3\a\d^2)=-3\b\d(\a\d-\b\c)$.
This is a contradiction, so $\b=0$. Then we have $\a^2\c\d=0$. Since $\a\d-\b\c = \a\d \neq 0$, we see $\c=0$.
Therefore $\s = \left(\begin{smallmatrix} \a & 0 \\ 0 & \d \end{smallmatrix}\right)$ and $\a^2=\d^2$.
We can check that $\det \s|_V=\a\d$ and $\hdet \s=\a^2\d^2$, so $\hdet \s=(\det \s|_V)^2$.

(Case 3) If $S=k\<x, y\>/(xy^2+yxy+y^2x+ax^3, yx^2+xyx+x^2y+ay^3)$ where $a\in k\setminus \{0, \pm 1,\pm 3\}$,
then
$\sfw_S=x^2y^2+xy^2x+y^2x^2+yx^2y+xyxy+yxyx + ax^4 +ay^4 \in \Sym^4V$.
Since $\overline \sfw_S=6x^2y^2+ax^4+ay^4$, if $\s= \left(\begin{smallmatrix} \a & \b \\ \c & \d \end{smallmatrix}\right) \in \GL(2, k)$ extends to $\s\in \GrAut S$, then
$\s(6x^2y^2+ax^4+ay^4)=\l (6x^2y^2+ax^4+ay^4)$
for some $\l\in k$.
Since
\begin{align*}
\s(6x^2y^2+ax^4+ay^4)
&= 6(\a x+\b y)^2(\c x+\d y)^2+a(\a x+\b y)^4+a(\c x+\d y)^4\\
&= (6\a^2\c^2+a\a^4+a\c^4)x^4
+4(3\a^2\c\d+3\a\b\c^2+a\a^3\b+a\c^3\d)x^3y\\
&\qquad +6(\a^2\d^2+4\a\b\c\d+\b^2\c^2+a\a^2\b^2+a\c^2\d^2)x^2y^2\\
&\qquad\qquad +4(3\b^2\c\d+3\a\b\d^2+a\a\b^3+a\c\d^3)xy^3
+(6\b^2\d^2+a\b^4+a\d^4)y^4,
\end{align*}
it follows that
\begin{align}
&3\a^2\c\d+3\a\b\c^2+a\a^3\b+a\c^3\d=0, \quad 3\b^2\c\d+3\a\b\d^2+a\a\b^3+a\c\d^3=0, \; \textrm{and} \label{eqhdet1}\\
&\a^2\d^2+4\a\b\c\d+\b^2\c^2+a\a^2\b^2+a\c^2\d^2 =a^{-1}(6\a^2\c^2+a\a^4+a\c^4)=a^{-1}(6\b^2\d^2+a\b^4+a\d^4). \label{eqhdet2}
\end{align}
By the equations in (\ref{eqhdet1}), we have $(3\a\b+a\c\d)(\a\d+\b\c)(\a\d-\b\c)=0$.

Assume that $3\a\b+a\c\d=0$. Then $(9-a^2)\a^2\c\d= (9-a^2)\b^2\c\d=0$ by (\ref{eqhdet1}).
Since $a\neq \pm 3$ and $\a\d-\b\c\neq 0$, it follows that either $\c=0$ or $\d=0$.
If $\c =0 $, then $\b=0$, so $\s = \left(\begin{smallmatrix} \a & 0 \\ 0 & \d \end{smallmatrix}\right)$ and $\a^2=\d^2$ by (\ref{eqhdet2}).
We can check that $\det \s|_V=\a\d$ and $\hdet \s=\a^2\d^2$, so $\hdet \s=(\det \s|_V)^2$.
If $\d =0 $, then $\a=0$, so $\s = \left(\begin{smallmatrix} 0 & \b \\ \c & 0 \end{smallmatrix}\right)$ and $\b^2=\c^2$ by (\ref{eqhdet2}).
We can check that $\det \s|_V=-\b\c$ and $\hdet \s=\b^2\c^2$, so $\hdet \s=(\det \s|_V)^2$.

Assume that $\a\d+\b\c=0$. Then $\a^4=\c^4, \b^4=\d^4$ by (\ref{eqhdet1}),
so $\s = \left(\begin{smallmatrix} \a & \b \\ -\xi\a & \xi\b \end{smallmatrix}\right)$ where $\xi$ is a 4th root of unity.
By (\ref{eqhdet2}), we see $(3\xi^2+a)(\a^4-\b^4)=0$. Since  $a\neq \pm 3$, it follows that $\b^2=\pm \a^2$.
Since $(2a-2\xi^2)\a^2\b^2 = a^{-1}(6\xi^2+2a)\a^4$ by (\ref{eqhdet2}),
we have
$a^2 +(1-\xi^2)a +3\xi^2=0$ or $a^2 -(1+\xi^2)a -3\xi^2=0$.
Since $\xi^2=\pm 1$ and $a \neq \pm 1, \pm 3$, we see $a = \pm \sqrt{-3}$.
Hence it follows that unless $a = \pm \sqrt{-3}$, $\hdet \s=(\det \s|_V)^2$ for every $\s\in \GrAut S$.

Conversely, suppose that $S=k\<x, y\>/(xy^2+yxy+y^2x+\sqrt{-3} x^3, yx^2+xyx+x^2y+\sqrt{-3} y^3)$.
If $\s = \left(\begin{smallmatrix} \a & \b \\ -\xi\a & \xi\b \end{smallmatrix}\right)$ where $\b^2=\pm \a^2, \xi^2=\mp 1$,
then we can check that $\s\in \GL(2, k)$ extends to $\s\in \GrAut S$, and $\det \s|_V=2\xi\a\b$, so $(\det \s|_V)^2=-4\a^4$.
Since
\begin{align*}
&\s(6x^2y^2+\sqrt{-3}x^4+\sqrt{-3}y^4)\\
&= \a^4(-2\sqrt{-3}\xi^2 +2)\sqrt{-3}x^4
+(-2\a^2\b^2\xi^2 + 2\sqrt{-3}\a^2\b^2)6x^2y^2
+\b^4(-2\sqrt{-3}\xi^2 +2)\sqrt{-3}y^4\\
&= \a^4(2 \pm 2\sqrt{-3})\sqrt{-3}x^4
+(2\a^4 \pm 2\sqrt{-3}\a^4)6x^2y^2
+\a^4(2 \pm 2\sqrt{-3})\sqrt{-3}y^4\\
&= \a^4(2 \pm 2\sqrt{-3})(6x^2y^2+\sqrt{-3}x^4+\sqrt{-3}y^4),
\end{align*}
we obtain $\hdet \s = \a^4(2 \pm 2\sqrt{-3}) = -4\a^4(\frac{-1 \mp \sqrt{-3}}{2}) = -4\a^4\om= (\det \s|_V)^2\om\neq (\det \s|_V)^2$ 
where $\om$ is a primitive 3rd root of unity. 

The homomorphism defined by $x\mapsto \sqrt {-1}x, y\mapsto y$ induces the isomorphism 
\begin{align*}
& k\<x, y\>/(xy^2+yxy+y^2x+\sqrt{-3} x^3, yx^2+xyx+x^2y+\sqrt{-3} y^3) \\
&\cong k\<x, y\>/(xy^2+yxy+y^2x-\sqrt{-3} x^3, yx^2+xyx+x^2y-\sqrt{-3} y^3), 
\end{align*}
so we have the result. 
\end{proof}

\section{Connections to Other Classes of Algebras} 

Let $S=J(\sfw)=T(V)/(R)$ be a 3-dimensional noetherian quadratic Calabi-Yau algebra.  In \cite {MS1}, it was shown that $c(\sfw)\not \in \Sym^3V$ if and only if $S$ is a deformation quantization of the polynomial algebra $k[x, y, z]$, and $c(\sfw)\in \Sym^3V$ if and only if $S$ is a Clifford algebra.  In this last section, we will see something similar holds at least in one direction.  

\subsection{Deformation Quantizations} 

For a scalar $\l\in k$ and a homogeneous polynomial $f\in k[x, y, z]$ of degree $3$ where $\deg x=\deg y=\deg z=1$, the algebra $S^{\l}_f:=k\<x, y, z\>/([y, z]-\l \widetilde {f_x}, [z, x]-\l \widetilde {f_y}, [x, y]-\l \widetilde {f_z})$ is a \emph{deformation quantization} of $k[x, y, z]$ (\cite {DM}), which is often a 3-dimensional noetherian quadratic Calabi-Yau algebra (\cite {MS1}).  
For a scalar $\l\in k$ and a weighted homogeneous polynomial $f\in k[x, y, z]$ of degree $4$ where $\deg x=\deg y=1$, $\deg z=2$, we may still define $S^{\l}_f:=k\<x, y, z\>/([y, z]-\l \widetilde {f_x}, [z, x]-\l \widetilde {f_y}, [x, y]-\l \widetilde {f_z})$.  We do not know if $S^{\l}_f$ is a deformation quantization of $k[x, y, z]$, but we will study when a 3-dimensional noetherian cubic Calabi-Yau algebra is of the form $S^{\l}_f$.  

\begin{theorem} \label{thm.dq}
Let $V$ be a 2-dimensional vector space and $\sfw\in V^{\otimes 4}$.  Fix a basis $x, y$ for $V$.  If $c(\sfw)\not \in \Sym^4V$, then $J(\sfw)=S^{\l}_f$ where  $\l=-\frac{3}{8\mu (c(\sfw))}\in k$ and $f=\overline {\sfw}+\frac{4\mu (c(\sfw))}{3}z^2\in k[x, y, z]$ with $\deg x=\deg y=1, \deg z=2$.
\end{theorem}

\begin{proof} Since $J(c(\sfw))=J(\sfw)$, we may assume that $\sfw=c(\sfw)\in \Im c$ is a superpotential.  If $V$ is a 2-dimensional vector space with a basis $x, y$, then we may write $\sfw=\a \sfw_1+\b \sfw_2+\sfw'$ where $\a, \b\in k$ and $\sfw'\in W'$, and, in this case,   
\begin{align*}
\partial_x\sfw=\a(xy^2+y^2x)+\b yxy+\partial_x\sfw', \quad
\partial_y\sfw=\a(yx^2+x^2y)+\b xyx+\partial_y\sfw'.
\end{align*}
Note that $\sfw\not \in \Sym^4V$ if and only if $\a\neq \b$, and, in this case,  we set $\l=\frac{3}{4(\a-\b)}$ and $f=\overline {\sfw}-\frac{2(\a-\b)}{3}z^2$.
Since $f=2(2\a+\b)x^2y^2+f'-\frac{2(\a-\b)}{3}z^2$ where $f'=\overline {\sfw'}$, it follows that $[x, y]-\l \widetilde {f_z} =xy-yx-z$.
Suppose that $[x, y]-\l \widetilde {f_z} =xy-yx-z=0$. Then, by Lemma \ref{lem.pd},
\begin{align*}
[y, z]-\l \widetilde {f_x} & =yz-zy-\frac{3}{4(\a-\b)}(4(2\a+\b)\widetilde {xy^2}+\widetilde {f'_x}) \\
& =y(xy-yx)-(xy-yx)y-\frac{3(2\a+\b)}{\a-\b}\frac{xy^2+yxy+y^2x}{3}-\frac{3}{4(\a-\b)}\widetilde {f'_x} \\
& =-(\frac{2\a+\b}{\a-\b}+1)(xy^2+y^2x)-(\frac{2\a+\b}{\a-\b}-2)yxy-\frac{3}{4(\a-\b)}4\partial _x\sfw' \\
& =-\frac{3}{\a-\b}(\a (xy^2+y^2x)+\b yxy+\partial _x\sfw') \\
& =-\frac{3}{\a-\b} \partial_x\sfw
\end{align*}
and, similarly, 
$[z, x]-\l \widetilde {f_y}=-\frac{3}{\a-\b} \partial_y\sfw$.
Since $\sfw=\a\sfw_1+\b\sfw_2+\sfw'=\frac{\b-\a}{2}(\sfw_2-\sfw_1)+\frac{\b+\a}{2}(\sfw_2+\sfw_1)+\sfw'$ where $\a, \b\in k$ and $\sfw'\in W'$, we have $\a-\b=-2\mu (\sfw)$, so $\l=-\frac{3}{8\mu (c(\sfw))}\in k$ and $f=\overline {\sfw}+\frac{4\mu (c(\sfw))}{3}z^2\in k[x, y, z]$.
\end{proof}

\subsection{Clifford Algebras} 

A 3-dimensional noetherian cubic AS-regular algebra is never a Clifford algebra in the usual sense.  We need to define a new class of algebras which look like Clifford algebras. 

\begin{definition} Let $M_1, \dots, M_n\in M_n(k)$ be square matrices of degree $n$ over $k$.
We define a graded algebra by 
$$A(M_1, \dots, M_n):=k\<x_1, \dots, x_n, y_1, \dots, y_n\>/(x_ix_j^2+x_jx_ix_j+x_j^2x_i-\sum _{k=1}^n(M_k)_{ij}y_k; \;  \textnormal {$y_k$ are central} )$$
with $\deg x_i=1, \deg y_i=3$.  
\end{definition} 

\begin{theorem} \label{thm.cli}
Let $V$ be a 2-dimensional vector space. If $\sfw\in \Sym^4V$ is a Calabi-Yau superpotential, then $J(\sfw)$ is isomorphic to $A(M_1, M_2)$ for some $M_1, M_2\in M_2(k)$.
\end{theorem} 

\begin{proof} If $M_1=\left(\begin{smallmatrix} 3 & -a \\ 0 & 0 \end{smallmatrix}\right), M_2=\left(\begin{smallmatrix} 0 & 0 \\ -b & 3 \end{smallmatrix}\right)\in M_2(k)$, then $A(M_1, M_2)$ is generated by $x, y, X, Y$ with the defining relations 
$$3x^3-3X, xy^2+yxy+y^2x+aX, yx^2+xyx+x^2y+bY, 3y^3-3Y; \; \textnormal { $X, Y$ are central.}$$
It follows that 
$$A(M_1, M_2)=k\<x, y\>/(xy^2+yxy+y^2x+ax^3, yx^2+xyx+x^2y+by^3; \; \textnormal { $x^3, y^3$ are central}).$$

On the other hand, if $\sfw\in \Sym^4V$ is a Calabi-Yau superpotential, then 
$$J(\sfw)\cong k\<x, y\>/(xy^2+yxy+y^2x+ax^3, yx^2+xyx+x^2y+by^3)$$ 
for some $a, b\in k$ by Proposition \ref{prop.symc}, 
so it is enough to show that $x^3, y^3$ are central in $J(\sfw)$.  In $J(\sfw)$, 
\begin{align*}
xy^3-y^3x & = -(yxy+y^2x+ax^3)y+y(xy^2+yxy+ax^3) = -ax^3y+ayx^3 \\
& =ax(yx^2+xyx+by^3)-a(xyx+x^2y+by^3)x = abxy^3-aby^3x,
\end{align*} 
so $(1-ab)(xy^3-y^3x)$.  By the classification in Proposition \ref{prop.symc}, $ab\neq 1$, so $xy^3=y^3x$ in $J(\sfw)$.  By symmetry, $yx^3=x^3y$ in $J(\sfw)$, so $x^3, y^3$ are central in $J(\sfw)$. 
\end{proof}

\medskip
\noindent\textit{Acknowledgment.} 
The second author is indebted to Yasuhiro Ishitsuka for providing valuable information on divisors in $\PP^1 \times \PP^1$.
The authors thank the referee for reading the manuscript carefully and providing useful suggestions.


\begin{thebibliography}{99}

\bibitem{AS}
M. Artin and W. Schelter, Graded algebras of global dimension 3,
{\it  Adv. Math.} {\bf 66} (1987), 171-216.

\bibitem{ATV}
M. Artin, J. Tate and M. Van den Bergh,
{\it Some algebras associated to automorphisms of elliptic curves},
The Grothendieck Festschrift, Vol. I, 33--85, {\it Progr. Math.}, {\bf 86},
Birkhauser, Boston, MA, 1990.

\bibitem{ATV2}
M. Artin, J. Tate and M. Van den Bergh,
Modules over regular algebras of dimension 3, 
{\it Invent. Math.} {\bf 106} (1991), 335--388.

\bibitem{B}
R. Bocklandt,
Graded Calabi Yau algebras of dimension 3, 
{\it J. Pure Appl. Algebra} {\bf 212} (2008), 14--32.

\bibitem{BSW} 
R. Bocklandt, T. Schedler and M. Wemyss, 
Superpotentials and higher order derivations, 
{\it J. Pure Appl. Algebra} {\bf 214} (2010), 1501--1522.

\bibitem{CKWZ}
K. Chan, E. Kirkman, C. Walton and J.J. Zhang,
Quantum binary polyhedral groups and their actions on quantum planes,
{\it J. Reine Angew. Math.} {\bf 719} (2016), 211--252.

\bibitem{DM} 
J. Donin and L. Makar-Limanov, 
Quantization of quadratic Poisson brackets on a polynomial algebra of three variables, 
{\it J. Pure Appl. Algebra} {\bf 129} (1998), 247-261. 

\bibitem{JZ}
P. J{\o}rgensen and J. J. Zhang,
Gourmet's guide to Gorensteinness,
{\it Adv. Math.} {\bf 151} (2000), 313--345.

\bibitem{KK}
E. Kirkman and J. Kuzmanovich,
Fixed subrings of Noetherian graded regular rings,
{\it J. Algebra} {\bf 288} (2005), 463--484. 

\bibitem{KKZ1}
E. Kirkman, J. Kuzmanovich and J. J. Zhang,
Gorenstein subrings of invariants under Hopf algebra actions,
{\it J. Algebra} {\bf 322} (2009), 3640--3669.

\bibitem{KKZ2}
E. Kirkman, J. Kuzmanovich, and J. J. Zhang,
Shephard-Todd-Chevalley Theorem for skew polynomial rings,
{\it Algebr. Represent. Theory} {\bf 13} (2010), 127--158.

\bibitem{MS2}
I. Mori and S. P. Smith, 
$m$-Koszul Artin-Schelter regular algebras, 
{\it J. Algebra} {\bf 446} (2016), 373--399. 

\bibitem{MS1}
I. Mori and S. P. Smith, 
The classification of Calabi-Yau algebras with 3 generators and 3 quadratic relations,
{\it Math. Z.} {\bf 287} (2017), 215--241. 

\bibitem{MU}
I. Mori and K. Ueyama,
Ample group actions on AS-regular algebras and noncommutative graded isolated singularities,
{\it Trans. Amer. Math. Soc.} {\bf 368} (2016), 7359--7383.

\bibitem{RRZ}
M. Reyes, D. Rogalski and J. J. Zhang, 
Skew Calabi-Yau algebras and homological identities, 
{\it Adv. Math.} {\bf 264} (2014), 308--354.

\end{thebibliography}
\end{document}